\documentclass{birkjour}
\usepackage{tikz-cd}
\usepackage{subfigure}
\usepackage{verbatim}
 \newtheorem{theorem}{Theorem}[section]
 \newtheorem{corollary}[theorem]{Corollary}
 \newtheorem{lemma}[theorem]{Lemma}
 \newtheorem{proposition}[theorem]{Proposition}
 \theoremstyle{definition}
 \newtheorem{definition}[theorem]{Definition}
 \newtheorem{notation}[theorem]{Notation}
 \theoremstyle{remark}
 \newtheorem{remark}[theorem]{Remark}
 \newtheorem{example}{Example}
 \numberwithin{equation}{section}

 \newcommand{\R}{\mathbb{R}}
 \newcommand{\C}{\mathbb{C}}
 \newcommand{\Z}{\mathbb{Z}}

 \newcommand{\EE}{\mathcal{E}}
 \newcommand{\RR}{\mathcal{R}}
  \newcommand{\M}{\mathcal{M}}
 \newcommand{\E}{\mathbf{E}}

 \newcommand{\fsep}{\hspace*{\fill}}

 \newcommand{\w}{\textcolor{black}}
 \newcommand{\p}{\textcolor{black}}
 \newcommand{\m}{\textcolor{black}}

\begin{document}

%
%
%

\title[Singular Improper Affine Spheres ]
 {Singular Improper Affine Spheres \\ from a given Lagrangian Submanifold}

\author[M. Craizer]{Marcos Craizer}

\address{%
Departamento de Matem\'{a}tica - PUC-Rio\br
Rio de Janeiro, RJ, 22453-900, 
Brazil}

\email{craizer@puc-rio.br}
\thanks{The first author thanks CNPq and the third author thanks Fapesp, for financial support during the preparation of this manuscript. The research of the second author was supported by NCN grant no. DEC-2013/11/B/ST1/03080 }

\author[W. Domitrz]{Wojciech Domitrz}
\address{Faculty of Mathematics and Information Science \br
Warsaw University of Technology \br
ul. Koszykowa 75, 00-662 Warszawa, 
Poland}
\email{domitrz@mini.pw.edu.pl}

\author[P. de M. Rios]{Pedro de M. Rios}
\address{Departamento de Matem\'atica - ICMC \br
Universidade de S\~ao Paulo \br
S\~ao Carlos, SP, 13560-970, 
Brazil}
\email{prios@icmc.usp.br}




\date{august 07, 2015}

\begin{abstract}
Given a Lagrangian submanifold $L$ of the affine symplectic $2n$-space, one can canonically and uniquely  define a center-chord and a special improper affine sphere of dimension $2n$, both of whose sets of singularities contain $L$. Although these improper affine spheres (IAS) always present other singularities away from $L$ (the off-shell singularities studied in \cite{CDR}), they may also present singularities other than $L$ which are arbitrarily close to $L$, the so called singularities ``on shell''. These on-shell singularities  possess a hidden $\mathbb Z_2$ symmetry that is absent from the off-shell singularities.  In this paper, we study these canonical IAS obtained from $L$ and their on-shell singularities, in arbitrary even dimensions, and classify all stable Lagrangian/Legendrian singularities on shell that may occur for these IAS when $L$ is a curve or a Lagrangian surface.
\end{abstract}

\maketitle

\section{Introduction}


An improper affine sphere (IAS) is a hypersurface whose affine Blaschke normal vectors are all parallel. They are given as the graph of a function $F:\R^m\to\R$
satisfying the Monge-Amp\`ere equation
\begin{equation}\label{eq:MA}
\det(D^2 F)=\pm 1.
\end{equation}
In the lowest dimensional cases, surfaces in $\R^3$, there are two classes of improper affine spheres: the convex ones, satisfying the equation $\det(D^2 F)=1$, and the non-convex ones, satisfying the equation $\det(D^2 F)=-1$.

For two planar curves $\alpha^{+}:U\subset\R\to\R^{2}$ and $\alpha^{-}:V\subset\R\to\R^2$, denote by
$x(u,v)=\frac{1}{2}(\alpha^{+}(u)+\alpha^{-}(v))$ the mid-point of $(\alpha^+(u),\alpha^-(v))$
and denote by $f(u,v)$ the area of the region bounded by the chord connecting $\alpha^{+}(u)$ and $\alpha^{-}(v)$, plus a chosen fixed chord $\xi_0$ connecting arbitrary points $\alpha^{+}(u_0)$ and $\alpha^{-}(v_0)$, plus the arcs of $\alpha^{+}$ and $\alpha^{-}$ between these two chords ($f\equiv f_{\xi_0}$ depends on the choice of $\xi_0$, of course, but for another choice $\xi_0'$, $f_{\xi_0'}-f_{\xi_0}=\mbox{constant}$).
The map $(u,v)\mapsto (x(u,v),f(u,v))$ in a non-convex IAS and, conversely, any $2$-dimensional non-convex IAS is locally as above, for certain curves $\alpha^{+}$ and $\alpha^{-}$. Since the mid-chord $y(u,v)=\frac{1}{2}(\alpha^{+}(u)-\alpha^{-}(v))$ is the symplectic gradient of $f$ \footnote{More precisely, $Y(x)=y(u,v)$ is the Hamiltonian vector field of $F(x)=f(u,v)$, for $x=x(u,v)$ the center as above, with respect to the canonical symplectic form on $\mathbb R^2\ni x$.}, this type of IAS \p{was called {\it center-chord} in \cite{CDR}, where this construction was generalized  to arbitrary even dimensions substituting the pair of planar curves by a pair $(L^{+},L^{-})$ of Lagrangian submanifolds of $\R^{2n}$. But in fact, this generalization was first presented in \cite{Cortes06}, where  IAS of this type were referred to as special para-K\"ahler manifolds.}

 The center-chord IAS is independent of parameterizations of the Lagrangian submanifolds and the singular set of a center-chord IAS is given by the pairs $(u,v)$ such that $T_uL^{+}$ and $T_vL^{-}$ are not transversal. The image of the singular set by the map $x(u,v)$ is the Wigner caustic of the pair $(L^{+},L^{-})$ and will be denoted $\E_{cc}(L^{+},L^{-})$, while the image of the singular set by the map $(x(u,v),f(u,v))$ will be denote $\tilde\E_{cc}(L^{+},L^{-})$.

For a holomorphic function $H:\C\to\C$, $H(z)=P(s,t)+iQ(s,t)$, where $z=s+it$, let us denote $x(s,t)=(s,\frac{\partial Q}{\partial t})$, $y(s,t)=(t,\frac{\partial Q}{\partial s})$, and also $f(s,t)=Q(s,t)-t\frac{\partial Q}{\partial t}$. Then, the map $(s,t)\mapsto (x(s,t),f(s,t))$ is a convex IAS whose symplectic gradient is $y$ \footnote{$Y(x)=y(s,t)$ is the Hamiltonian vector field of $F(x)=f(s,t)$, for $x=x(s,t)$ as above.}. Conversely, any $2$-dimensional convex IAS is as above, for a certain holomorphic function $H$. This construction was generalized to arbitrary even dimensions in \cite{Cortes00}, by considering holomorphic maps $H:\C^n\to\C$, \p{and IAS of this type were shown to be special K\"ahler manifolds in the sense of \cite{Freed}. Thus, this type of IAS was called {\it special} in \cite{Cortes00}.}

The singular set of a special IAS is given by the pairs $(s,t)$ such that $\frac{\partial^2Q}{\partial t^2}$ is singular. The image of the singular set by the map $x(s,t)$ is a caustic and will be denoted $\E_{sp}(\mathbb L)$, while the image of the singular set by the map $(x(s,t),f(s,t))$ will be denoted $\tilde\E_{sp}(\mathbb L)$, where $\mathbb L$ is the graph of $dH$ in $\mathbb C^n\times\mathbb C^n$.

For both the center-chord and the special IAS, the function $F:\R^{2n}\to \R$, given by $F(x)=f(u,v)$, satisfies the Monge-Amp\`ere equation (\ref{eq:MA}), but generically each such solution $F$ of the Monge-Amp\`ere equation has singularities, as studied in \cite{CDR}. On the other hand, what was not explored in \cite{CDR} and is the object of the present paper is that, in various instances, a subset of the singular set of $F$ is a Lagrangian submanifold $L\subset\R^{2n}$.

In fact, by taking the same Lagrangian submanifold, $L^+=L^-=L$, we obtain an interesting subclass of the center-chord improper affine spheres.
In this case, $L$ is contained in the {\it Wigner caustic} $\E_{cc}(L)$ of $L$. The study of the Wigner caustic of $L$ is of some interest in physics (\cite{DMR},\cite{DR}), and this subclass of the center-chord IAS is also of interest in computational vision  (\cite{Craizer08},\cite{Giblin04}).

In this paper we introduce the corresponding subclass for special IAS. This subclass consists of special IAS defined by holomorphic maps $H:\C^n\to\C$ that takes the real space $\R^n$ into the real line $\R$, which implies that the real function $Q$, above, is an odd function of $t$.
Denote by $L$ the image of the real space $\R^n$ by the map $x$, which is a Lagrangian submanifold of $\R^{2n}$. Since the holomorphic map
$H$ can be recovered from $L$, we shall denote by $\E_{sp}(L)$ the corresponding caustic of the special IAS. As in the center-chord case, $L$ is contained in
$\E_{sp}(L)$. In \cite{Craizer08}, this type of IAS was considered for $n=1$.

Generically, the sets $\E_{cc}(L)$ and $\E_{sp}(L)$ contain $L$ and other points away from $L$, but they
may also contain more points than just $L$ in any neighborhood of $L$, the so-called {\it on-shell} part of $\E_{cc}(L)$ and $\E_{sp}(L)$, denoted by $\E_{cc}^s(L)$ and $\E_{sp}^s(L)$, respectively. In \cite{CDR}, singularities of $\E_{cc}(L)\setminus \E_{cc}^s(L)$ and $\E_{sp}(L)\setminus\E_{sp}^s(L)$, also called  off-shell singularities, were studied and classified. In this paper, we shall study and classify the singularities of $\E_{cc}^s(L)$ and $\E_{sp}^s(L)$, and of their Legendrian analogues $\tilde{\E}_{cc}^s(L)$ and $\tilde{\E}_{sp}^s(L)$.

Since both $\E_{cc}^s(L)$ and $\E_{sp}^s(L)$ are \p{sets of critical values} of Lagrangian maps (respectively Legendrian maps for $\tilde{\E}_{cc}^s(L)$ and $\tilde{\E}_{sp}^s(L)$), it is natural
to study them in this context \p{using generating functions and generating families.}
\p{But the study of singularities via generating families is of a local nature, so we shall actually study germs of singularities. In this setting, the on-shell singularity-germs of these IAS are described by the following theorem, which is detailed in section \ref{genfamilydescription} below and generalize \cite[Theorem 2.11]{DMR}.}

For center-chord IAS, \p{if $L$ is locally generated by function $S=S(q)$ via}
\begin{equation}
L=\{(q,p)\in\R^{2n}|\ p=\m{d}S\} \ , \nonumber
\end{equation}
then a generating \p{family} for $\E_{cc}^s(L)$ in a neighborhood of $L$ is given by
\begin{equation}\label{oddGcc}
G_{cc}^s(\beta,q,p)=\frac{1}{2}\left( S(q+\beta)-S(q-\beta) \right)-p\cdot\beta \ .
\end{equation}

\p{In the special case, for the function $Q=Q(s,t)$ introduced above, with
	\begin{equation}
	L=\big\{(q,p)=\Big(s,\frac{\partial Q}{\partial t}(s,0)\Big), s\in\mathbb R^n	\big\} \nonumber
	\end{equation}
locally, a generating family for $\E_{sp}^s(L)$  in a neighborhood of $L$ is given by
\begin{equation}\label{oddGsp}
G_{sp}^s(\beta,q,p)=Q(q,\beta)-p\cdot\beta \ .
\end{equation}
\begin{theorem}\label{precise}
The germ at $0\in L$ of $\E_{cc}^s(L)$, resp. $\E_{sp}^s(L)$, is the set of critical values the Lagrangian map-germ $\pi|_{\mathcal{L}}:\mathcal L\to\mathbb R^{2n}$, where $\mathcal L$ is the Lagrangian submanifold-germ of $(T\mathbb R^{2n},\Omega)$ determined by the generating family $G_{cc}^s$, resp. $G_{sp}^s$, as above (cf. (\ref{lag1}) below). Similarly, the germ at  $0\in L\times\{0\}$ of $\tilde{\E}_{cc}^s(L)$, resp. $\tilde{\E}_{sp}^s(L)$, is the set of critical values of the Legendrian map-germ $\tilde\pi|_{\tilde{\mathcal{L}}}:\tilde{\mathcal L}\to\mathbb R^{2n}\times\mathbb R$, where $\tilde{\mathcal L}$ is the Legendrian submanifold-germ of $(T\mathbb R^{2n}\times\mathbb R,\theta)$ determined by the generating family $G_{cc}^s$, resp. $G_{sp}^s$, as above (cf. (\ref{leg1}) below).
\end{theorem}}

\p{In the above theorem, $\pi:T\mathbb R^{2n}\to \mathbb R^{2n}$ is the canonical projection, with $\tilde\pi:T\mathbb R^{2n}\times\mathbb R\to \mathbb R^{2n}\times\mathbb R$ its trivial extension, and $\Omega$ is the tangential lift (\cite{Tu},\cite{GU}) of the canonical symplectic form on $\mathbb R^{2n}$, cf. (\ref{Omega}), with $\theta$ being the contact form associated to $\Omega$ which is semi-basic w.r.t. $\tilde{\pi}$, cf. (\ref{theta}) below.}
	
\p{From (\ref{oddGcc}), $G_{cc}^s$ is odd in $\beta$.  Likewise, from (\ref{oddGsp}) and the fact that $Q(q,\beta)$ is odd in $\beta$, $G_{sp}^s$ is also odd in $\beta$. Therefore, in both cases we must consider odd deformations of odd generating families, for classification of the singularities on shell, cf. section 4 and specially Theorem  \ref{equivalence} below.}

	\p{The set $\E_{cc}^s(L)$ was studied in \cite{DMR} \w{(see also \cite{DZ3})}, where its stable Lagrangian singularities were classified, when $L$ is a curve or a surface. In this paper, we adapt the results from
	\cite{DMR} to
classify the stable Legendrian singularities on shell of the center-chord IAS and
classify the stable Lagrangian and Legendrian singularities on  shell of the special IAS, for such IAS associated to curves and Lagrangian surfaces.} 

Thus, the important characteristic of these on-shell singularities, both in  center-chord and  special cases, is that they possess a hidden $\mathbb{Z}_2$ symmetry \p{that descends from the explicit $\mathbb Z_2$ symmetry of $\mathcal L$ or $\tilde{\mathcal{L}}$ due to $G$ being odd in $\beta$. Such a  hidden $\mathbb{Z}_2$ symmetry  is absent from the off-shell singularities, so the classification of on-shell singularities is different.}  This is relevant for the geometric  study of solutions of the Monge-Amp\`ere equation for functions $F:\R^{2n}\to\R$  whose singularity set contains a Lagrangian submanifold $L\subset\R^{2n}$.

On the other hand, a remarkable distinction of the special case is that $G^s_{sp}$ is necessarily real analytic, and thus, in the special case, we should in principle consider the classification
of singularities up to analytic equivalence. It turns out, however, that imposing real analyticity amounts to no refinement on the classification under smooth diffeomorphisms \cite{BKP}, so the classification of singularities of odd smooth functions, presented in \cite{DMR}, applies equally to classification of singularities on shell of both  center-chord and  special IAS.

This paper is organized as follows:

In section 2, we review the properties of center-chord and special IAS, emphasizing the ones
obtained from a single Lagrangian submanifold; we also characterize the IAS coming from a single Lagrangian submanifold among all possible center-chord and special IAS, showing that for each $L$ there is a unique canonical center-chord IAS and a unique canonical special IAS obtained from $L$. \p{Thus, in particular, many results in this section are of a global nature, in contrast with the remaining of the paper.}

In section 3, we describe the on-shell odd generating families for the canonical center-chord and special IAS, \p{detailing Theorem \ref{precise}, and then, in section 4, we study equivalence and stability of Lagrangian/Legendrian maps which are $\mathbb Z_2$-symmetric along the fibers, in relation with  equivalence, stability and versality of generating families  in the odd category}, expanding and complementing the treatment developed in \cite{DMR}. Section 4 is also intended to clarify many results in singularity theory which are not so familiar to nonspecialists, in view of the interdisciplinary nature of the paper.

Then, in section 5,  we present the classification of simple Lagrangian and Legendrian singularities on shell
for the canonical center-chord and special IAS obtained from $L$, producing explicit examples that show they are all realized,
and we also present the classification and geometrical condition/interpretation of all stable Lagrangian and Legendrian singularities on shell
for the canonical center-chord and special IAS
obtained from curves and Lagrangian surfaces, by adapting the results presented in \cite{DMR}.

Finally, in Section 6 we present the proof of Theorem \ref{equivalence}.

\section{Singular center-chord and special IAS }

\subsection{Center-chord IAS}\label{sectionccias}

Let $U, V$ be \w{simply connected} open subsets of $\R^n$. Consider a pair of Lagrangian immersions $\alpha^{+}:U\to\R^{2n}$ and $\alpha^{-}:V\to\R^{2n}$, where $\R^{2n}$ is the affine symplectic space with the canonical symplectic form $\omega=\sum_{i=1}^n dq_i\wedge dp_i$, $q=(q_1,...,q_n)$, $p=(p_1,...,p_n)$.
Define $x,y:U\times V\to\R^{2n}$ by
\begin{equation}
x(u,v)=\frac{1}{2}\left( \alpha^{+}(u)+\alpha^{-}(v)\right),\ \ y(u,v)=\frac{1}{2}\left( \alpha^{+}(u)-\alpha^{-}(v) \right).
\end{equation}

Fix a pair of parameters $(u_0,v_0)\in U\times V$.  For a given $(u,v)\in U\times V$, consider the oriented curve $\delta(u,v,u_0,v_0)$ in $\R^{2n}$ obtained by concatenating
the chord connecting $\alpha^{+}(u_0)$ and $\alpha^{-}(v_0)$, a curve in $L^{-}=\alpha^-(V)$ connecting $\alpha^{-}(v_0)$ and $\alpha^{-}(v)$,  the chord connecting $\alpha^{-}(v)$ and $\alpha^{+}(u)$, and
a curve in $L^{+}=\alpha^+(U)$ connecting $\alpha^{+}(u)$ and $\alpha^{+}(u_0)$.

Because $\delta$ is closed and $\R^{2n}$ is simply connected, $\delta=\partial\Sigma$ and \w{ $\omega=d\eta$, where $\eta= \sum_{i=1}^n p_idq_i$}. Denote
$$f[\delta] \ =\int_{\Sigma}\omega \ =\int_{\delta}\eta \ =\int_{\delta(u,v,u_0,v_0)} p_idq_i \ .$$
 It is clear that $f[\delta]$ is independent of the choice of the curves along $L_{+}$ and $L_{-}$ and that, if we change the initial pair $(u_0,v_0)$ we just add a constant to $f$, \w{since $U, V$ are simply connected and $\alpha^+, \alpha^-$ are Lagrangian immersions}. Thus, up to a constant,
we can write $f[\delta]=f(u,v)$. The following proposition was proved in \cite{CDR}:

\begin{proposition}\label{CCIAS}
The function $f:U\times V\to\R$ defined above  \w{ up to an additive constant}  satisfies
$$
f_u=\omega(x_u,y),\ \ f_v=\omega(x_v,y).
$$
Moreover, the map $\phi:U\times V\subset\R^{2n}\to\R^{2n+1}$
$
\phi(u,v)=\left(x(u,v),f(u,v) \right)
$
is an improper affine map, which, when regular, defines an improper affine sphere (IAS).
\end{proposition}

\begin{remark}
	In all cases of IAS considered in this paper, we shall often abbreviate and refer to singularities of the map \m{$\phi$} as singularities of the IAS.
\end{remark}

The type of IAS from Proposition \ref{CCIAS} is called a {\it center-chord} IAS. By a smooth change of coordinates, we may assume that locally
\begin{equation}\label{eq:IASccGraph}
\alpha^{+}(u)= \left( u, dS^{+}(u)  \right), \ \ \alpha^{-}(v)= \left( v, dS^{-}(v)  \right).
\end{equation}
for some pair of  functions $S^{+}:U\subset\R^{n}\to\R$ and $S^{-}:V\subset\R^{n}\to\R$.

The {\it singular set} of the center-chord IAS consists of the pairs $(u,v)\in U\times V$ such that the tangent spaces of $\alpha^{-}(u)$ and  $\alpha^{+}(v)$ are not transversal subspaces of $\R^{2n}$, or equivalently, $d^2S^{+}(u)-d^2S^{-}(v)$ is singular. The image of the singular set by the map $x(u,v)$ is called the {\it Wigner caustic} of $(L^{-},L^{+})$ and will be denoted  by $\E_{cc}(L^{-},L^{+})$, while the image of the singular set by $\phi=(x,f)$ will be denoted
$\tilde{\E}_{cc}(L^{-},L^{+})$.

\medskip
\paragraph{Center-chord IAS from a given Lagrangian submanifold}
In this paper, we shall be particularly interested in the case that the Lagrangian submanifolds $L^{+}$ and $L^{-}$ coincide, i.e.,
$$
\alpha^{+}(u)=\alpha(u), \ \ \ \alpha^{-}(v)=\alpha(v),
$$
for some Lagrangian immersion $\alpha:U\subset\R^n\to\R^{2n}$. In this case we shall denote the image of this immersion by $L=L^{+}=L^{-}$ and the
corresponding IAS by $\phi_{cc}(L)$. In case $\alpha$ is of the form \eqref{eq:IASccGraph} we shall write
\begin{equation}\label{eq:GraphS}
\alpha(u)=\left( u,  dS(u)\right),
\end{equation}
for some function $S:U\subset\R^n\to\R$.

When $L^{+}=L^{-}=L$, the caustic $\E_{cc}(L,L)$ is the Wigner caustic of the Lagrangian submanifold $L$ and will be denoted $\E_{cc}(L)$.
In this case, the set $u=v$ is contained in the singular set of $\phi$.
Since \m{$x(u,u)=\alpha(u)$}, we conclude that $L\subset\E_{cc}(L)$ (see also \cite{DMR} and \cite{DR}).

\begin{example}\label{ex:TorusIn}
Assume that $\alpha(u)=(\cos(u),\sin(u))$, i.e., $L$ is the unit circle in the plane. Then
$$
x(u,v)=\cos\left( \frac{u-v}{2}   \right)\left(  \cos\left( \frac{u+v}{2}   \right),\sin\left( \frac{u+v}{2}   \right)\right)
$$
$$
y(u,v)=\sin\left( \frac{u-v}{2}   \right)\left( - \sin\left( \frac{u+v}{2}   \right),\cos\left( \frac{u+v}{2}   \right)\right)
$$
$$
f(u,v)=\frac{1}{4}\left(v-u+\sin(u-v)   \right),
$$
(see figure \ref{fig:CircleIn}). The image of the map $x$ is the unit disc $D$ and the singular set is $u=v+k\pi$, $k\in\Z$.
This example can be generalized by taking
$$
\alpha(u)=\left(\cos(u_1),\sin(u_1),.....,\cos(u_n),\sin(u_n)\right), \ \ u=(u_1,...,u_n),
$$
so that $L$ is the $n$-dimensional torus in $\R^{2n}$. Then
$$
f(u,v)=\frac{1}{4}\sum_{i=1}^n \left(v_i-u_i+\sin(u_i-v_i)   \right).
$$
The singular set of this center-chord IAS is the union of submanifolds
$$
\R^2\times\R^2.....\times \{u_i=v_i+k\pi\}\times ...\R^2, \ \ 1\leq i\leq n.
$$
\end{example}

\begin{figure}[htb]
\centering
\includegraphics[width=0.50\linewidth]{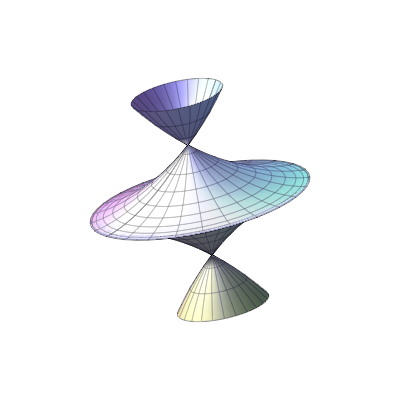}
\caption{ The singular center-chord IAS of example \ref{ex:TorusIn} }
\label{fig:CircleIn}
\end{figure}

\subsection{Special IAS}\label{sectionspecialias}

Consider a complex Lagrangian immersion $\gamma=(\gamma_1,\gamma_2):W\subset\C^n\to\C^{2n}$ and denote its image by $\mathbb L$.
Also, let $x,y:W\to\R^{2n}$ be given by
\begin{equation}\label{eq:IASSpGeneral}
x(w)=\frac{1}{2}\left( \gamma(w)+\bar\gamma(w) \right);\ \ y(w)=\frac{1}{2i}\left( \gamma(w)-\bar\gamma(w) \right),
\end{equation}
$w\in W$, $w=u+iv$. The following proposition was proved in \cite{CDR}:

\begin{proposition}\label{SIAS}
There exists $f:W\to\R$, unique up to an additive constant, such that
\begin{equation}\label{eq:fufv}
f_u=\omega(x_u,y),\ \ f_v=\omega(x_v,y)
\end{equation}
and the map
\begin{equation*}\label{eq:}
\phi(u,v)=\left(x(u,v),f(u,v) \right)
\end{equation*}
is an improper affine map, which, when regular, defines an improper affine sphere (IAS).  Moreover, the IAS $\phi$ does not depend on the parameterization of the complex Lagrangian immersion $\gamma$.
\end{proposition}

From equations (\ref{eq:IASSpGeneral}) and (\ref{eq:fufv}), the function $f$ can be given a geometrical description similar to the one in the center-chord case. Let $\delta=\delta(w,w_0)$ be a oriented curve formed by the concatenation of the imaginary chord joining $\bar\gamma(w_0)$ to  $\gamma(w_0)$, a curve in $\mathbb L$ joining $\gamma(w_0)$ to $\gamma(w)$, the imaginary chord joining $\gamma(w)$ to $\bar\gamma(w)$, and a curve in $\overline{\mathbb L}$ joining $\bar\gamma(w)$ to $\bar\gamma(w_0)$.

Because $\delta$ is closed and $\C^{2n}$ is simply connected, $\delta=\partial\Sigma$, and because $\omega$ is real and exact, $\omega=d\eta$, where $\eta=\frac{i}{2}w d\bar{w}$. Denote
$$f[\delta] \ =\int_{\Sigma}\omega \ =\int_{\delta}\eta \ = \frac{i}{2}\sum_{k=1}^n\int_{\delta(w,w_0)} w_kd\bar{w}_k \ .$$
It is clear that the real function $f[\delta]$ is independent of the choice of the curves along $\mathbb L$ and $\overline{\mathbb L}$ (Lagrangian condition) and that, if we change the initial point $w_0$ we just add a constant to $f$. Thus, up to a constant,
we can write $f[\delta]=f(w)$.

\medskip\noindent
The type of IAS from Proposition \ref{SIAS}  is called {\it special} (\cite{Cortes00}).  As shown in \cite{CDR}, by a holomorphic change of coordinates, we may locally reparameterize $\gamma$ by
\begin{equation}\label{eq:IASSpGraph}
\gamma(z)=\left(z,dH(z)\right), \ \ z\in Z,
\end{equation}
where $H:Z\subset\C^n\to\C$ is a holomorphic map. Furthermore, setting $z=s+it$ and $H=P+iQ$, we can write
\begin{equation}\label{eq:IASSpGraph2}
x(z)=\left(s,\frac{\partial Q}{\partial t}\right),\ \ y(z)=\left(t,\frac{\partial Q}{\partial s}\right),
\end{equation}
and we also have that, up to an additive constant,
\begin{equation}\label{eq:IASspf}
f(s,t)=Q(s,t)-\displaystyle{\sum_{k=1}^n} t_k\cdot\frac{\partial Q}{\partial t_k} \ ,
\end{equation}
\p{a formula first obtained by Cort\'es in \cite{Cortes02}.}

The {\it singular set} of the special IAS consists of points $(s,t)\in U$ such that $\frac{\partial^2Q}{\partial t^2}$ is singular. The image of the singular set by $x$ will be denoted $\E_{sp}(\mathbb L)$, while the image of the singular set by $(x,f)$ will be denoted $\tilde\E_{sp}(\mathbb L)$.

\begin{remark}
	Recall that, in both the center-chord and the special cases, the regularity of the map $\phi:(u,v)\mapsto(x(u,v),f(u,v))$ is equivalent to having an invertible map $(u,v)\mapsto x(u,v)$, $x^{-1}$ possibly multiple valued, so that each function $F=f\circ x^{-1}:\mathbb U\subset\R^{2n}\to\R$ is well defined and  satisfies the Monge-Amp\`ere equation \eqref{eq:MA}, and each regular branch of $\phi$ is a graph of $F$.
\end{remark}

\medskip
\paragraph{Special IAS from a given Lagrangian submanifold}
In this paper we shall be interested in the case $Z$ is a domain in $\C^n$ invariant by conjugation and $\bar\gamma(z)=\gamma(\bar{z})$, which
is equivalent to saying that $H(\R^n\cap Z)$ is contained in $\R$. In this case, we shall denote by $L$ the image of $Z\cap\R^n$ by the map $x(z)$.

Assume now we are given a real analytic Lagrangian submanifold $L$, image of $\alpha$
given by equation \eqref{eq:GraphS}, for some $S:U\subset\R^n\to\R$ real analytic.
\m{ Then there exists a domain $Z\subset\C^n$ invariant by conjugation such that $Z\cap\R^n=U$ and a holomorphic map $H: Z\to\C$ such that $H$ restricted to $U$ equals $S$. }
In particular, the image of $Z\cap\R^n$ by $H$ is contained in $\R$.

It is clear from the above two paragraphs that $\mathbb L\cap\overline{\mathbb L}=L\subset\R^{2n}$ is a Lagrangian submanifold of the real symplectic space and that the special IAS defined by $\gamma(z)=(z,dH(z))$ depends only on $L$. Therefore we shall denote it by $\phi_{sp}(L)$, and we shall denote by $\E_{sp}(L)$ the caustic of $\phi_{sp}(L)$.

We may assume that $L$ is given by equation \eqref{eq:IASSpGraph}, for some
$H:Z\to\C$ satisfying $\bar{H}(z)=H(\bar{z})$. This implies that the imaginary part $Q$ of $H$ is odd in $t$, where $z=s+it$, and so
$\frac{\partial^2Q}{\partial t^2}=0$ for $t=0$. We conclude that $L\subset\E_{sp}(L)$.

\begin{example}\label{ex:TorusOut}
Let  $\gamma(z)=(\cos(z),\sin(z))$. Then $L$ is the unit circle in $\R^2$ and
$$
x(s,t)=\cosh(t)\left(\cos(s),\sin(s)\right);\ \ y(s,t)=\sinh(t)\left(-\sin(s),\cos(s)\right),
$$
$$
f=\frac{1}{4}\left( \sinh(2t)-2t \right),
$$
(see figure \ref{fig:CircleOut}). The image of the map $x$ is the complement $\bar{D}$ of the open unit disc,
while the singular set is $t=0$. This example can be generalized by considering $\gamma:\C^n\to\C^{2n}$ given by
$$
\gamma(z)=\left(\cos(z_1),\sin(z_1),....,\cos(z_n),\sin(z_n)\right), \ \ z=(z_1,...,z_n).
$$
In this case $L$ is the $n$-dimensional torus $L\subset\R^{2n}$ and
$$
f(s,t)=\frac{1}{4}\sum_{i=1}^n \left( \sinh(2t_i)-2t_i\right),
$$
the caustic \p{is} the union of submanifolds $\R^2\times\R^2.....\times \{t=0\}\times ...\R^2$, where $\{t=0\}$ is
in coordinate $i$.
\end{example}

\begin{figure}[htb]
\centering
\includegraphics[width=0.50\linewidth]{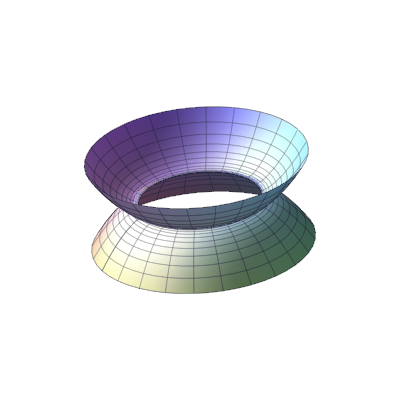}
\caption{The singular special IAS of example \ref{ex:TorusOut}.  }
\label{fig:CircleOut}
\end{figure}

\subsection{Affine Bj\"orling problem}

The {\it affine Bj\"orling problem} for $n=1$ consists in finding an improper affine sphere containing a smooth curve, analytic in the convex case, with a prescribed co-normal along it.
\m{ Recall that the co-normal at a point of the curve is a co-vector that has the tangent line in its kernel and takes the value $1$ at the affine Blaschke normal}.
We observe that the co-normal for both types of IAS is given by $(y,1)$ and thus this problem is equivalent to finding an IAS given the values of $(x,f)$ and $y$ along a curve in the parameter plane. 

The affine Bj\"orling problem for $n=1$ has a unique solution for the center chord case (see \cite[Thm.3.1]{Milan13} and  \cite[Thm.6.1]{Milan14}) and also for the special case (see \cite[Thm.6.1]{Galvez07}). We shall see below that by taking $y=0$ along the curve $L$, we obtain the IAS $\phi_{cc}(L)$ and
$\phi_{sp}(L)$ in each case.

We now let $n$ be general and characterize $\phi_{cc}(L)$ among the center-chord IAS $\phi_{cc}(L^{-},L^{+})$ and $\phi_{sp}(L)$ among the special IAS $\phi_{sp}(\mathbb{L})$.

For a center-chord IAS $\phi=\phi_{cc}(L^{-},L^{+})$, denote by \p{$\E_{cc}^{0}(L^{-},L^{+})$} the subset of $\E_{cc}(L^{-},L^{+})$ such that
the tangent spaces to $L^{+}$ at $\alpha^+(u)$ and to $L^{-}$ at $\alpha^-(v)$ are strongly parallel. For $\alpha^{\pm}$ of the form (\ref{eq:IASccGraph}), this is equivalent to having
$d^2S^{+}(u)=d^2S^{-}(v)$.

\begin{proposition}
Let $\phi=\phi_{cc}(L^{-},L^{+})$ and  $L$ a Lagrangian submanifold of $\R^{2n}$. The following statements are equivalent:
\begin{enumerate}
\item $\phi=\phi_{cc}(L)$.
\item \p{$L\subset\E_{cc}^{0}(L^{-},L^{+})$} and  $f$ is constant along $L$.
\item $y=0$ along $L$.
\end{enumerate}
\end{proposition}

\begin{proof}
\noindent
$(1)\Rightarrow(2)$ follows from the description of $\phi_{cc}(L)$ given in section \ref{sectionccias}.

\smallskip\noindent
$(2)\Rightarrow(3)$:
If $f$ is constant along $L$, then necessarily $\omega(y,x_u)=\omega(y,x_v)=0$,
implying that the chord $y$ is tangent to $L$. On the other hand, \p{$L\subset\E_{cc}^{0}(L^{-},L^{+})$}  implies that the tangent
spaces of $\alpha^{+}(u)$ and $\alpha^{-}(u)$ are strongly parallel. We conclude that the tangent spaces in fact coincide and $y=0$.

\smallskip\noindent
$(3)\Rightarrow (1)$: The condition $y=0$ at $L$ implies that $u=v$ and $dS^{+}(u)=dS^{-}(v)$. Thus $L$ is contained in the image of the diagonal $u=v$ and $dS^{+}(u)=dS^{-}(u)$. This implies that, up to a constant, $S^{+}(u)=S^{-}(u)$ and so $\phi=\phi_{cc}(L)$.
\end{proof}


\begin{remark}\label{yinA} If $L\supset A$, where $A$ is an affine subspace of $\R^{2n}$, then condition (2) only implies that $L=L^+=L^-$, thus implying (1), since it is possible to have a nonvanishing $y\in TA$ if $\alpha^+\neq\alpha^-$. But then, by choosing the canonical parametrization $\alpha^+=\alpha^-$ for $L^+=L^-$, we obtain (3). This canonical choice when $L^+=L^-$ shall always be assumed.  \end{remark}

For a special IAS $\phi=\phi_{sp}(\mathbb{L})$, denote by \p{$\E_{sp}^{0}(\mathbb{L})$} the subset of $\E_{sp}(\mathbb{L})$ such that the tangent spaces to $\mathbb{L}$ at $\gamma(z)$
and to $\overline{\mathbb{L}}$ at $\bar\gamma(z)$ are strongly parallel. For $\gamma$ of the form (\ref{eq:IASSpGraph}), this is equivalent to having
$\frac{\partial^2Q}{\partial t^2}=0$.

\begin{proposition}
Let $\phi=\phi(\mathbb{L})$ and  $L$ a Lagrangian submanifold of $\R^{2n}$. The following statements are equivalent:
\begin{enumerate}
\item $\phi=\phi_{sp}(L)$.
\item \p{$L\subset\E_{sp}^{0}(\mathbb{L})$} and $f$ is constant along $L$.
\item $y=0$ along $L$.
\end{enumerate}
\end{proposition}

\begin{proof}
\noindent
$(1)\Rightarrow(2)$ follows from the description of $\phi_{sp}(L)$ given in section \ref{sectionspecialias}.	

\smallskip\noindent	
$(2)\Rightarrow(3)$:
If $f$ is constant along $L$, then necessarily $y$ is tangent to $L$. We may assume $t=t(s)$ along $L$ with $t(0)=0$. Differentiating equation \eqref{eq:IASSpGraph2}(a), we obtain that along $L$
$$
x_s=\left( 1, \frac{\partial^2Q}{\partial t^2}t_s+\frac{\partial^2Q}{\partial s\partial t}  \right).
$$
Since $\frac{\partial^2Q}{\partial t^2}=0$ along $L$ and $y$ is tangent to $L$ we conclude  from  \eqref{eq:IASSpGraph2}(b) that
$$
t\frac{\partial}{\partial t}\frac{\partial Q}{\partial s}=\frac{\partial Q}{\partial s},
$$
which implies that $\frac{\partial Q}{\partial s}=ct$, for some constant $c$. Since $Q$ is harmonic, $\frac{\partial^2 Q}{\partial s^2}=0$, and so $t(s)=0$, which
implies $y=0$.

\smallskip\noindent
$(3)\Rightarrow(1)$:
$y=0$ implies $t=0$ and $\frac{\partial Q}{\partial s}=0$ at $L$. Thus $L$ is contained in the image of the parameterization
$\left(s,\frac{\partial Q}{\partial t}(s,0)\right)$ and $\frac{\partial Q}{\partial s}(s,0)=0$. Thus we know $\frac{\partial Q}{\partial s}$ and
$\frac{\partial Q}{\partial t}$ along the curve $t=0$. This implies that we know $dH$ at $t=0$. So we know
$H$, up to an additive constant, which implies $\phi=\phi_{sp}(L)$.
\end{proof}

\begin{remark}\label{goodpar}
	For $\phi_{sp}(L)$, the choice of parametrization  (\ref{eq:IASSpGraph}) with  $\bar{H}(z)=H(\bar{z})$ is canonical in the sense that $L\subset\mathbb L$ and $L\subset\overline{\mathbb L}$ have the same parametrization (cf. remark \ref{yinA}) and therefore this is implicitly assumed.
\end{remark}

\begin{remark}\label{LinIAS}
For both $\phi_{cc}(L)$ and $\phi_{sp}(L)$, since $f$ is constant along $L$, we can choose $f$ such that $f(L)=0$. With this canonical choice, $L\subset \tilde\E_{cc}(L)$ and $L\subset \tilde\E_{sp}(L)$, and we have lost the freedom of adding a constant to $f$.
\end{remark}

In view of the above remarks, we present the following:

\begin{definition}
	With canonical choices outlined in remarks \ref{yinA}, \ref{goodpar} and \ref{LinIAS}, we call $\phi_{cc}(L)$ and $\phi_{sp}(L)$ the two {\bf canonical IAS} obtained uniquely from $L$.
\end{definition}

\section{Description of Lagrangian/Legendrian singularities on shell for the two canonical IAS from a Lagrangian submanifold}\label{genfamilydescription}

In this section we describe $\E_{cc}$ and $\E_{sp}$ as \p{sets of critical values of} Lagrangian maps (caustics)
and $\tilde{\E}_{cc}$ and $\tilde{\E}_{sp}$
as \p{sets of critical values of} Legendrian maps.
However, \p{because our following descriptions will be 
of a local nature,} it's necessary to distinguish two different ``parts'' of each of these sets. 

The off-shell part of $\E_{cc}(L)$  is locally of the form $\E_{cc}(L^-,L^+)$, where $L^-\neq L^+$ are germs of $L$ at two distinct points in $L$, \p{but the on-shell part of $\E_{cc}(L)$ is locally of the form $\E_{cc}(L',L')$, where $L'$ is the germ of $L$ at one point in $L$. Similarly for the other sets above} 
\footnote{The local characterization of \p{the off-shell part of $\E_{cc}(L)$} and the local classification of its singularities can be found in \cite{DR}, \cite{CDR} \w{(global properties of the Wigner caustic $\E_{cc}(L)$  of planar curves are studied in \cite{DZ1}-\cite{DZ3}).}}.

\p{From now on, we are concerned} with describing and classifying the ``on-shell'' part of the sets above and their singularities. 
\p{Because our treatment will be local, it is tempting to define a germ of $\E_{cc}(L)$ on shell, for instance, as the germ of $\E_{cc}(L)$ at a point $x\in L$. However, although} in various instances the on-shell and the off-shell parts of these sets are disconnected  \footnote{\p{This is often the case for $\E_{cc}(L)$  of closed convex planar curves} \w{(for general closed planar curves the on-shell part is composed of curves connecting two inflection points \cite{DZ3})}.}, \p{there are various other instances when  the germ of $\E_{cc}(L)$ at a point $x\in L$ is the union of a germ of $\E_{cc}(L)$  on-shell and a germ of $\E_{cc}(L)$  off-shell \footnote{\p{Recall that $x\in\E_{cc}(L)$-off-shell if $x$ is the midpoint of a chord connecting distinct points $a,b\in L$ s.t. $T_aL$ and $T_bL$ are not transversal. But it may happen that $x\in L$.}}. For this reason, the most precise description of a germ of $\E_{cc}(L)$  (or $\E_{sp}(L)$, $\tilde{\E}_{cc}(L)$, $\tilde{\E}_{sp}(L)$) on-shell is the one given by Theorem \ref{precise}, which we detail below.}

\begin{notation} Because  we'll focus on the on-shell parts of $\E_{cc}(L)$, or $\E_{sp}(L)$, and  $\tilde{\E}_{cc}(L)$, or $\tilde{\E}_{sp}(L)$, the on-shell parts of these sets shall be denoted with a superscript ``s'' as $\E_{cc}^s(L)$, or $\E_{sp}^s(L)$, and  $\tilde{\E}_{cc}^s(L)$, or $\tilde{\E}_{sp}^s(L)$, respectively, 
or  simply by  $\E^s(L)$ and  $\tilde\E^s(L)$, when the kind  is not specified.
\end{notation}

\subsection{Generating functions and families}

\p{Let $x=(q,p)=(q_1,...,q_n,p_1,...,p_n)$, $y=(\dot{q},\dot{p})=(\dot{q}_1,...,\dot{q}_n,\dot{p}_1,...,\dot{p}_n)$, denote by $\pi:T\mathbb R^{2n}=\R^{2n}\times\R^{2n}\to\R^{2n}$ the canonical projection $\pi(x,y)=x$ and by $\tilde \pi:T\R^{2n}\times\R\ni(q,p,\dot{q},\dot{p},z) \mapsto(q,p,z)\in\R^{2n}\times\R$ its trivial extension.}

\p{Denote by $\omega=\sum_{i=1}^n dq_i\wedge dp_i$
the canonical symplectic form on $\R^{2n}$, by
\begin{equation}\label{Omega}
\Omega=\sum_{i=1}^n dq_i\wedge d\dot{p}_i +d\dot{q}_i \wedge dp_i
\end{equation}
the
tangential lift (\cite{Tu},\cite{GU})  of $\omega$ on $T\R^{2n}$ and by
\begin{equation}\label{theta}
\theta=dz+\sum_{i=1}^n \dot{q}_i dp_i - \dot{p}_i dq_i
\end{equation}
the contact form on $T\R^{2n}\times\R$ associated to $\Omega$ and semi-basic w.r.t. $\tilde{\pi}$}.

Let $U$ be an open subset of $\R^{2n}$. We shall denote by $\mathcal{L}$ the image of the Lagrangian immersion
$\m{L=}(x,y):U\rightarrow (T\mathbb R^{2n},\Omega)$ and
by $\tilde{\mathcal{L}}$ the image of the Legendrian immersion
$\m{\tilde{L}=}(x,y,f):U\rightarrow (T\mathbb R^{2n}\times \mathbb R,\{\theta=0\})$.
We are interested in studying  the singularities of the Lagrangian map
$\pi \circ (x,y)$ and the Legendrian map $\tilde \pi\circ (x,y,f)$, where $x=(q,p)$, $y=(\dot{q},\dot{p})$.

The main tools used for classifying Lagrangian and Legendrian singularities are the generating functions and
generating families.
A generating function of the Lagrangian submanifold $\mathcal L$ and the Legendrian submanifold $\tilde {\mathcal L}$ is a function
$$
g:\R^{n}\times\R^{n}\ni (q,\dot{q})\mapsto g(q,\dot{q})\in\R,
$$
satisfying
\begin{equation}\label{eq:defineGF1}
\mathcal L=\{ (q,p,\dot{q},\dot{p}):   \frac{\partial g}{\partial q}=\dot{p}, \frac{\partial g}{\partial \dot{q}}=p \}.
\end{equation}
and
\begin{equation}\label{eq:defineGF2}
\tilde {\mathcal L}=\{ (q,p,\dot{q},\dot{p},z): \frac{\partial g}{\partial {q}}=\dot{p}, \frac{\partial g}{\partial \dot{q}}=p, z=g(q,\dot{q})-\dot{q}\cdot p \}.
\end{equation}
A generating family of the Lagrangian map $\pi \circ L$ and the Legendrian map $\tilde \pi \circ \tilde L$ is a function
$G:\R^{n}\times\R^{2n}\ni (\beta,q,p) \mapsto G(\beta,q,p)\in \R$ satisfying
\begin{equation}\label{lag1}
\mathcal L=\{ (q,p,\dot{q},\dot{p}): \exists \beta : \frac{\partial G}{\partial \beta}=0,  \frac{\partial G}{\partial q}=\dot{p}, -\frac{\partial G}{\partial p}=\dot{q} \}.
\end{equation}
and
\begin{equation}\label{leg1}
\tilde{\mathcal L}=\{ (q,p,\dot{q},\dot{p},z): \exists \beta : \frac{\partial G}{\partial \beta}=0,  \frac{\partial G}{\partial q}=\dot{p}, -\frac{\partial G}{\partial p}=\dot{q}, z=G(\beta,q,p) \}.
\end{equation}
A generating family can be obtained from a generating function by
\begin{equation}\label{gen-fam-gen-fun}
G(\beta,q,p)=g(q,\beta)- p\cdot\beta.
\end{equation}

However, we stress that generating families are local objects, suitable for local descriptions and classifications, therefore we now focus on the generating families for $\phi_{cc}(L)$ and $\phi_{sp}(L)$ on shell,  when $(p,q)$ is in a neighborhood of \p{$L\ni 0$, in order to complete the statement of Theorem \ref{precise}}.

\medskip
\paragraph{Generating families for center-chord and special IAS on shell}

For a center-chord IAS $\phi_{cc}(L)$, where $L$ is defined by $(u, dS(u))$,
straightforward calculations show that
$$
g_{cc}^s(q,\dot{q})=\frac{1}{2}S(q+\dot{q})-\frac{1}{2}S(q-\dot{q})
$$
is a generating function on shell and so
\begin{equation}\label{eq:GenFamCC}
G_{cc}^s(\beta,q,p)=\frac{1}{2}S(q+\beta)-\frac{1}{2}S(q-\beta)-p\cdot\beta.
\end{equation}
is a generating family for $\phi_{cc}(L)$ on shell. For a special IAS defined by the holomorphic function $H$ taking $\R^n$ to $\R$, \p{by straightforward calculations}
$$
g_{sp}^s(q,\dot{q})=Q(q,\dot{q})
$$
is \p{shown to be} a generating function on shell and the generating family for $\phi_{sp}(L)$ on shell is given by
\begin{equation}\label{eq:GenFamSP}
G_{sp}^s(\beta,q,p)=Q(q,\beta)-p\cdot\beta,
\end{equation}
where $Q$  is the imaginary part of $H$.

\p{Equations (\ref{Omega})-(\ref{eq:GenFamSP}) complete the statement of Theorem \ref{precise}, whose proof follows straightforwardly from standard theory of Lagrangian and Legendrian maps. We refer to  \cite{Arnold} and subsection \ref{oddgenfamilygerm} below, for details.}

But note that both $G_{cc}^s(\beta,q,p)$ and $G_{sp}^s(\beta,q,p)$ are odd functions of $\beta$.

\begin{remark}
	This odd property of the generating families $G$ implies that the  singularities of 
	$\E^s(L)$, resp. $\tilde{\E}^s({L})$, 
	possess  a {\it hidden} $\Z_2$-symmetry, which descends from  an explicit fibred-$\Z_2$-symmetry of the Lagrangian, resp. Legendrian, submanifold ${\mathcal L}$, resp. $\tilde{\mathcal L}$, obtained from $G$ by (\ref{lag1}), resp. (\ref{leg1}). 
	\end{remark}
	
\p{Thus, in order to classify the singularities on shell of $\phi_{cc}(L)$ and $\phi_{sp}(L)$ we must consider: (i) the classification problem for  fibred-$\mathbb Z_2$-symmetric Lagrangian and Legendrian singularities; in relation with: (ii) the classification problem for generating families in the odd category.}	

\p{Part (ii) was carried out in \cite{DMR}, but there neither part (i) nor the relation between the two parts was treated. These are carried out in the next section.}

\section{Fibred-$\mathbb Z_2$-symmetric germs of  Lagrangian maps and their odd generating families}\label{singtheory}

In this section, \p{first we develop
the definition of Lagrangian stability in the context of fibred-$\mathbb Z_2$-symmetric Lagrangian map-germs (cf. Definition \ref{d3b}), by carefully} adapting the nonsymmetric treatment presented in \cite{Arnold}. Here we shall only work  in the Lagrangian setting, the extension to the Legendrian setting being straightforward because all Legendrian immersions we consider are graph-like (cf. \cite[Section 5.3]{Izu15}).

\p{Then, we detail the relation between Definition \ref{d3b} and the definition of stability of odd generating families (cf. Definition \ref{oddstablegen} and Corollary \ref{mc}). The central result for relating these two definitions is given by Theorem \ref{equivalence}, which relates the corresponding notions of equivalence (cf. Definitions \ref{d3} and \ref{foe}). Because the proof of this theorem is not too short, it has been placed in section \ref{App}, at the end of the paper.}

\p{Finally, we state, explaining its proof, the theorem which relates stability of odd generating families to their odd versality (cf. Theorem \ref{oddstablegen2} below), which will be used to classify the stable singularities on shell of the center-chord and special IAS, in the next section.}

We start by recalling basic definitions of the theory of Lagrangian singularities (cf. eg. \cite[Part III]{Arnold}) and then specialize some of these basic definitions to the fibred-$\mathbb Z_2$-symmetric context.

\subsection{Fibred-$\mathbb Z_2$-symmetric Lagrangian map-germs}

Let $M$ be a smooth (or analytic) $m$-dimensional manifold. Let $E\rightarrow M$ be a smooth (or analytic) fiber bundle over $M$.
A diffeomorphism of $E$ is {\it fibered} (or {\it fibred}) if it maps fibers to fibers.

Let $T^{\ast}M$ be the cotangent bundle of $M$ and let $\omega$ be the canonical symplectic form on $T^{\ast}M$. A smooth (or analytic) section $s:M\rightarrow T^{\ast}M$  is called Lagrangian if $s^{\ast}\omega=0$. Sections of $T^{\ast}M$ are differential $1$-forms on $M$  and it is easy to see that a section is Lagrangian iff the $1$-form is closed.
Thus, any germ of a smooth (or analytic)  Lagrangian section can be described as the differential of a smooth (or analytic) function-germ on $M$.  This function-germ is called a {\it generating function} of the germ of a Lagrangian section.

Let $\lambda$
and
$(\kappa,\lambda)=(\kappa_1,\cdots,\kappa_m,\lambda_1,\cdots,\lambda_m)$  be local coordinates on $M$ and  $T^{\ast}M$, respectively, then $\omega=\sum_{i=1}^m d\kappa_i\wedge d\lambda_i$. $(T^{\ast}M,\omega)$ with canonical projection $\pi: T^{\ast}M\ni (\kappa,\lambda) \mapsto \lambda\in M$ is a {\it Lagrangian fibre bundle}.

Let $L$ be a Lagrangian submanifold of $T^{\ast}M$ i.e. $\dim L=\dim M$ and the pullback of the symplectic form $\omega$ to $L$ vanishes.
Then $\pi|_ L:L\rightarrow  M$ is called a {\it Lagrangian map}. The
set of critical values of a Lagrangian map is called a {\it caustic}.
Let $L$ and $\tilde L$ be two Lagrangian submanifolds of $(T^{\ast}M,\omega)$:

\begin{definition}\label{d1}
	Two Lagrangian maps $\pi|_ L: L\rightarrow M$ and $\pi|_{\tilde L}:\tilde L\rightarrow M$ are {\it Lagrangian equivalent} if there exists a fibered symplectomorphism of   $(T^{\ast}M,\omega)$ mapping $L$ to $\tilde L$. A Lagrangian map is {\it stable} if every nearby Lagrangian map (in the Whitney topology) is Lagrangian equivalent to it. Likewise for germs of Lagrangian submanifolds and Lagrangian maps.
\end{definition}

We are interested in studying a special type of Lagrangian maps, this type consisting of maps which are $\mathbb Z_2$-symmetric in the fibers.

\begin{definition}\label{d2}
	A Lagrangian submanifold $L$ of $T^{\ast}M$ is {\it fibred-$\mathbb Z_2$-symmetric}  if
	for every point $(\kappa,\lambda)$ in $L$ the point $(-\kappa,\lambda)$ belongs to $L$. The Lagrangian map  $\pi|_ L:L\rightarrow  M$ is  {\it fibred-$\mathbb Z_2$-symmetric} if the Lagrangian submanifold $L$ is fibred-$\mathbb Z_2$-symmetric. Likewise for {\it fibred-$\mathbb Z_2$-symmetric} germs of Lagrangian submanifolds and Lagrangian maps.
\end{definition}
\begin{remark}
	\p{Because the Lagrangian map-germs 
		are fibred-$\mathbb Z_2$-symmetric, the $\mathbb Z_2$  symmetry is \emph{hidden} in their caustics. Thus, the  classification of caustics of fibred-$\mathbb Z_2$-symmetric Lagrangian map-germs differs from classifications of caustic-germs which are explicitly $\mathbb Z_2$-symmetric (cf. e.g. \cite{JR} for the latter).}
\end{remark}

It is easy to see that the fibers of $T^{\ast}M$ and the zero section of $T^{\ast}M$ are fibred-$\mathbb Z_2$-symmetric Lagrangian submanifolds.
We will study singularities of fibred-$\mathbb Z_2$-symmetric Lagrangian map-germs. Thus we need a Lagrangian equivalence which preserves this $\mathbb Z_2$-symmetry.
Let us denote  
$$
\zeta:T^{\ast}M\ni (\kappa,\lambda)=(-\kappa,\lambda)\in T^{\ast}M.
$$
The map $\zeta$ defines a $\mathbb Z_2$-action on $T^{\ast}M$ i.e. $\mathbb Z_2\cong\{\zeta, id_{T^{\ast}M}\}$.
\begin{definition}
	A fibred symplectomorphism   $\Phi$ of  $(T^{\ast}M,\omega)$ is {\it odd}  or {\it $\mathbb Z_2$-equivariant} if $\Phi\circ \zeta=\zeta\circ \Phi$.
\end{definition}

Since our consideration is local we may assume that $M=\mathbb R^m$.
A fibred symplectomorphism-germ $\Phi$ of  $(T^{\ast}M,\omega)$ has the form $\Phi=(\phi)^{\ast} + dG$, where $\phi:M\rightarrow M$ is a diffeomorphism-germ and $G$ is a smooth (analytic) function-germ on $M$ (see  \cite{Arnold} section 18.5).
Since $\zeta$ is the identity on the zero section of $T^{\ast}M$, odd fibred symplectomorphisms map the zero section  to itself.
This implies the following characterization of  odd fibred symplectomorphism-germs.
\begin{proposition}\label{odd-base}
	If $\Phi$ is an odd fibred symplectomorphism-germ of  $(T^{\ast}M,\omega)$ then $\Phi$ has the form $\Phi=(\phi)^{\ast}$, where $\phi:M\rightarrow M$ is a diffeomorphism-germ.
\end{proposition}
It is easy to see that odd fibred  symplectomorphisms map fibred-$\mathbb Z_2$-symmetric Lagrangian submanifolds to fibred-$\mathbb Z_2$-symmetric Lagrangian submanifolds. Thus we can define a $\mathbb Z_2$-symmetric Lagrangian equivalence of fibred-$\mathbb Z_2$-symmetric Lagrangian map-germs in the following way:

\begin{definition}\label{d3}
	Fibred-$\mathbb Z_2$-symmetric Lagrangian map-germs $\pi|_ L: L\rightarrow M$ and $\pi|_{\tilde L}:\tilde L\rightarrow  M$ are {\it $\mathbb Z_2$-symmetrically Lagrangian equivalent} if there exists an odd fibred symplectomorphism-germ of   $(T^{\ast}M,\omega)$  mapping $L$ to $\tilde L$.
\end{definition}

\p{This induces the following natural definition:} 

\begin{definition}\label{d3b}
A fibred-$\mathbb Z_2$-symmetric Lagrangian map-germ is {\it $\mathbb Z_2$-symmetrically stable} if every nearby fibred-$\mathbb Z_2$-symmetric Lagrangian map-germ (in Whitney topology) is $\mathbb Z_2$-symmetrically Lagrangian equivalent to it.
\end{definition}

Thus, definitions \ref{d3} and \ref{d3b} specialize in a natural way the definitions of Lagrangian equivalence and stability  (cf. Definition \ref{d1}) to the context of fibred-$\mathbb Z_2$-symmetric Lagrangian map-germs.

\subsection{Odd generating family-germs}\label{oddgenfamilygerm}

Any Lagrangian submanifold-germ can be described by a generating family. In the case of fibred-$\mathbb Z_2$-symmetric Lagrangian submanifold-germs the generating family can be odd in $\beta$.
Before we prove the above statement we introduce  some preparatory definitions (see \cite[Section 19.2]{Arnold}).

Because all the following descriptions are local, we take $M=\R^m$. Then, 
a bundle  $\rho :\mathbb R^n\times \R^m\ni (\beta,\lambda)\mapsto \lambda \in \mathbb R^m$  is called an {\it auxiliary bundle}, for which the space $\mathbb R^{n+m}=\mathbb R^n\times \mathbb R^m$ is the {\it big space} and $M=\mathbb R^m$ is the {\it base}.  The cotangent bundle of the big space is  the {\it big phase space} and the cotangent bundle of the base, $\pi: T^{\ast}\mathbb R^m\rightarrow \mathbb R^m$, is  the {\it small phase space}.

The {\it mixed space} $A$  for the auxiliary bundle $\rho$ is the set of elements of the big phase space which annihilate vectors tangent to the fibers of $\rho$. The {\it mixed bundle}  is the bundle over the big space induced from the small phase space by the map $\rho$. It is easy to see that the total space of the mixed bundle $\rho^{\ast}\pi$ is $A$ and the fibers of $\rho^{\ast}\pi$ are isomorphic to the fibers of $\pi$. $A$ is also the total space of the bundle $\pi^{\ast} \rho: A\rightarrow T^{\ast} \mathbb R^{n+m}$  induced from the auxiliary bundle $\rho$ by  $\pi$.  These bundles are described by the following diagrams.

\begin{center}
	\begin{tikzcd}
		T^{\ast}\mathbb R^m  \arrow[d,"\pi"]
		& A \arrow[r, "\iota", hook] \arrow[l, "\pi^{\ast}\rho"]  \arrow[d, "\rho^{\ast}\pi" ]  & T^{\ast}\mathbb R^{n+m} \arrow[dl] \\
		\mathbb R^m
		& \mathbb R^{n+m} \arrow[l, "\rho" ]
	\end{tikzcd}
	\begin{tikzcd}
		(\kappa,\lambda)  \arrow[d,"\pi",mapsto]
		& (\kappa,\beta,\lambda) \arrow[r, "\iota",mapsto] \arrow[l, "\pi^{\ast}\rho",mapsto]  \arrow[d, "\rho^{\ast}\pi",mapsto ]  & (0,\kappa,\beta,\lambda) \arrow[dl,mapsto] \\
		\lambda
		& (\beta,\lambda) \arrow[l, "\rho",mapsto ]
	\end{tikzcd}
\end{center}

\

A Lagrangian submanifold of the big phase space is called {\it $\rho$-regular} if it is transversal to the mixed space $A$ for $\rho$. The image  of the intersection of a $\rho$-regular Lagrangian submanifold with the mixed space $A$ by the natural projection $\pi^{\ast}\rho$ to the small phase space is a Lagrangian (immersed) submanifold and every germ of a Lagrangian submanifold of the small space can be obtained by this construction from the germ of $\rho$-regular Lagrangian section of the appropriate big phase space (see  \cite[Section 19.3]{Arnold}).

A function $F$ is a generating function of the Lagrangian section $\mathcal L$ of the big phase space if  $\mathcal L$ is described by 
\begin{equation}\label{genFunctBig}
\mathcal L=\left\{(\alpha,\kappa,\beta,\lambda)\in T^{\ast}\mathbb R^{n+m}|  \  \alpha=\frac{\partial F}{\partial \beta}(\beta,\lambda) , \  \kappa=\frac {\partial F}{\partial \lambda}(\beta,\lambda)\right\}, 
\end{equation}
\p{where we use the notation $\frac {\partial F}{\partial \lambda}(\beta,\lambda)=\left(\frac {\partial F}{\partial \lambda_1}(\beta,\lambda),\cdots, \frac {\partial F}{\partial \lambda_m}(\beta,\lambda)\right)$, etc.}  
Since the mixed space $A$ is described by $\{(\alpha,\kappa,\beta,\lambda)\in T^{\ast}\mathbb R^{n+m}| \alpha=0\}$, it follows that $L=\pi^{\ast}\rho(\mathcal L\cup A)$ is described by
\begin{equation}\label{genFam}
L=\left\{(\kappa,\lambda)\in T^{\ast}M \ \big| \ \exists \ \beta , \  \ \frac{\partial F}{\partial \beta}(\beta,\lambda) = 0, \ \  \kappa=\frac {\partial F}{\partial \lambda}(\beta,\lambda)\right\}.
\end{equation}
\p{The family of generating functions $F_{\beta}(\lambda)=F(\beta,\lambda)$ is called a {\it generating family} of the Lagrangian submanifold-germ $L\subset (T^{\ast}M,\omega)$ described by (\ref{genFam}). Although somewhat counterintuitive, one usually refers to $\beta$ as the {\it variables} and to $\lambda$ as the {\it parameters} of the generating family.} 


\begin{remark}\label{map_by_F}
	The {\it set of critical points of the family} F is  the following set
	$$
	\Sigma(F)=\left\{(\beta,\lambda)\in \mathbb R^{n+m} \ \big|  \ \frac{\partial F}{\partial \beta}(\beta,\lambda)=0 \right\} \ .
	$$
	Since the Lagrangian submanifold $\mathcal L$ is $\rho$-regular, $\Sigma(F)$ is a $m$-dimensional submanifold of $\mathbb R^{n+m}$. The set of critical points of the family $F$ is naturally diffeomorphic to the germ of the Lagrangian submanifold $L$ of the small phase space determined by the germ of the generating family $F$.
	Then the  Lagrangian map-germ in terms of the generating family $F$ is described by
	$$
	\Sigma(F) \ni (\beta,\lambda)\mapsto \lambda \in \mathbb R^m.
	$$
\end{remark}

\begin{proposition}\label{thisodd}
	If a Lagrangian submanifold-germ $L$ of $T^{\ast}M$  is fibred-$\mathbb Z_2$-symmetric then there exists a local generating family $F=F(\beta,\lambda)$ which is odd (in variables) i.e. $F(-\beta,\lambda)\equiv -F(\beta,\lambda)$.
\end{proposition}

\begin{proof}
	We use the method described in \cite{Arnold} (see  Example 6 in Section 18.3 and Section 19.3 C). There exist subsets  $J=\{j_1,\cdots,j_n\}$, $I=\{i_1,\cdots,i_{m-n}\}$ of $\{1,\cdots, m\}$ such that $I\cap J=\emptyset$ and $I\cup J=\{1,\cdots,m\}$  and a local  generating function $S=S(\kappa_J,\lambda_I)$ of $L$, where $\kappa_J=(\kappa_{j_1},\cdots,\kappa_{j_n})$ and $\lambda_I=(\lambda_{i_1},\cdots,\lambda_{i_{m-n}})$.
	Then, $L$ is locally described by 
	\begin{equation}\label{genFunct}
	L=\left\{(\kappa,\lambda)\in T^{\ast}M| \ \lambda_J=-\frac {\partial S}{\partial \kappa_J}(\kappa_J,\lambda_I), \  \kappa_I=\frac {\partial S}{\partial \lambda_I}(\kappa_J,\lambda_I) \right\}.
	\end{equation}
	
	Since $L$  is fibred-$\mathbb Z_2$-symmetric, if $(\kappa,\lambda)\in L$ then $(-\kappa,\lambda)\in L$. Hence, if
	$\lambda_J=-\frac {\partial S}{\partial \kappa_J}(\kappa_J,\lambda_I), \kappa_I=\frac {\partial S}{\partial \lambda_I}(\kappa_J,\lambda_I)$ then
	$\lambda_J=-\frac {\partial S}{\partial \kappa_J}(-\kappa_J,\lambda_I),- \kappa_I=\frac {\partial S}{\partial \lambda_I}(-\kappa_J,\lambda_I)$.  Thus we get
	\begin{equation}\label{S_odd}
	\frac {\partial S}{\partial \kappa_J}(-\kappa_J,\lambda_I)\equiv \frac {\partial S}{\partial \kappa_J}(\kappa_J,\lambda_I), \ \ \
	\frac {\partial S}{\partial \lambda_I}(-\kappa_J,\lambda_I)\equiv -\frac {\partial S}{\partial \lambda_I}(-\kappa_J,\lambda_I).
	\end{equation}
	
	The generating function-germ is determined up to an additive constant. So we may assume that $S(0,0)=0$. From (\ref{S_odd}) we obtain that $S=S(\kappa_J,\lambda_I)$ is an odd function-germ in $\kappa_J$.
	Consider a function-germ on a big space $\mathbb R^{n+m}$ of the form $F(\beta,\lambda)\equiv S(\beta, \lambda_I)+\langle\beta,\lambda_J\rangle$, where $\langle\cdot,\cdot\rangle$ is the dot product.
	Then $F=F(\beta,\lambda)$ is odd in $\beta$. It is easy to see that $F$ is a generating function of a Lagrangian section $\mathcal L$ of the big phase space $T^{\ast}\mathbb R^{n+m}$ described by (\ref{genFunctBig}) and $\mathcal L$ is $\rho$-regular. The set $\pi^{\ast}\rho(\mathcal L\cup A)$ is exactly $L$. Indeed, $\frac{\partial F}{\partial \beta}(\beta,\lambda)\equiv \frac {\partial S}{\partial \kappa_J}(\beta,\lambda_I)+\lambda_J$,
	$\frac{\partial F}{\partial \lambda_I}(\beta,\lambda)\equiv \frac {\partial S}{\partial \lambda_I}(\beta,\lambda_I)$ and  $\frac{\partial F}{\partial \lambda_J}(\beta,\lambda)\equiv \beta$.
	Thus, by (\ref{genFunct}) $L$ is locally described by (\ref{genFam}).
\end{proof}
\begin{remark}\label{pathological}
	We can choose such sets $J, \ I$ such that $n=\sharp J$ is the dimension of the kernel of the \m{differential of the} Lagrangian map $L\hookrightarrow T^{\ast}M \rightarrow M$. The Lagrangian map-germ  is described  in terms of $S$ in the following way
	$$
	\mathbb R^m \ni (\kappa_J, \lambda_I)\mapsto (-\frac {\partial S}{\partial \kappa_J}(\kappa_J,\lambda_I),\lambda_I).
	$$
	The coordinates $\lambda_J$ and $\kappa_J$ are called {\it pathological}. The arguments $\kappa_J$ are $n$ {\it pathological arguments}  of the function $S$.
\end{remark}

\subsection{$\mathbb Z_2$-symmetric Lagrangian stability and $\RR^{odd}$-versality}

From classical results (cf. \cite[Section 19.4]{Arnold}), in the non-symmetric context we know that Lagrangian equivalence of Lagrangian map-germs corresponds to stably fibred $\mathcal R^+$-equivalence of their generating families. 

\w{ We recall that to two germs of generating families $F, G: (\mathbb R^n\times \mathbb R^m,0)\rightarrow \mathbb R$ are fibred $\mathcal R^+$-equivalent if there is a fibred diffeomorphism-germ  $ \Psi(\beta,\lambda)\equiv (\Phi(\beta,\lambda),\Lambda(\lambda))$ and a function-germ $h:(\mathbb R^m,0) \rightarrow \mathbb R $ such that $F(\beta,\lambda)\equiv G(\Phi(\beta,\lambda),\Lambda(\lambda))+h(\lambda)$. Then, $F: (\mathbb R^k\times \mathbb R^m,0)\rightarrow \mathbb R$ and  $ G: (\mathbb R^l\times \mathbb R^m,0)\rightarrow \mathbb R$, with $k\ne l$, are stably  fibred $\mathcal R^+$-equivalent if there exist nondegenerate quadratic forms $Q_i:\mathbb R^{r_i}\rightarrow \mathbb R$ for $i=1,2$, s.t. $k+r_1=l+r_2=n$ and $F+Q_1, G+Q_2: (\mathbb R^{n}\times \mathbb R^m,0)\rightarrow \mathbb R$ are  fibred $\mathcal R^+$-equivalent.}

\w{In $\mathbb Z_2$-symmetric  context, the zero-section is preserved by $\mathbb Z_2$-symmetric Lagrangian equivalence and quadratic forms are not odd functions. Then},   
denote the group of diffeomorphism-germs  $(\mathbb R^n\times \mathbb R^m,0)\rightarrow (\mathbb R^n\times \mathbb R^m,0)$ by $\mathcal D_{n+m}$  and let
$\mathcal D_n^{odd}$ denote the subgroup of odd   diffeomorphism-germs  $(\mathbb R^n,0)\rightarrow (\mathbb R^n,0)$ i.e. $\Phi\in \mathcal D_n^{odd}$
if $\Phi(-\beta)\equiv-\Phi(\beta)$.

\begin{definition}\label{foe}\label{oe}
	Odd generating family-germs $F, G$ of fibred-$\mathbb Z_2$-symmetric Lagrangian submanifold-germs are {\em fibred  $\mathcal R^{odd}$-equivalent} if
	there exists an odd (in variables) fibred diffeomorphism-germ $\Psi \in \mathcal
	D_{n+m}$, that is, 
	$$ \Psi(\beta,\lambda)\equiv
	(\Phi(\beta,\lambda),\Lambda(\lambda)) \ , \ \text{with} \ \Phi|_{\mathbb R^n\times
		\{\lambda\}} \in \mathcal D_n^{odd} \  , \  \forall\lambda\in\R^m ,$$ such that
	\begin{equation}\label{GoPsi}
	F=G\circ \Psi.
	\end{equation}
	\end{definition}
\begin{remark}
In the notation of section 	\ref{genfamilydescription}, equation (\ref{GoPsi}) can be written as
$$
G(\beta,x)=\bar{G}(\bar\beta,\bar{x}) 
$$
(parameters $\lambda=x=(q,p)\in\R^{2n}=\R^m$), w.r.t. an odd fibred diffeomorphism-germ  denoted as  
$$
(\beta,{x})= \left(  \beta(\bar\beta,\bar{x}), x(\bar{x})   \right) \ .
$$
\end{remark}

\begin{proposition}
	If the generating families $G$ and $\bar{G}$ are fibred $\RR^{odd}$-equivalent, then the caustics $\E^s(L)$ and $\E^s(\bar{L})$ are diffeomorphic.
\end{proposition}
\begin{proof}
	Since
	$$
	\frac{\partial G}{\partial \beta}=\frac{\partial \bar{G}}{\partial \bar\beta}\frac{\partial \bar\beta}{\partial \beta}
	$$
	we conclude that $\frac{\partial G}{\partial \beta}=0$ if and only if $\frac{\partial \bar{G}}{\partial \bar\beta}=0$. Moreover
	$$
	\frac{\partial^2G}{\partial \beta^2}=\frac{\partial^2 \bar{G}}{\partial \bar\beta^2}\left(\frac{\partial \bar\beta}{\partial \beta}\right)^2+\frac{\partial \bar{G}}{\partial \bar\beta}\frac{\partial^2 \bar\beta}{\partial \beta^2}.
	$$
	Thus $\frac{\partial^2 G}{\partial \beta^2}=\frac{\partial G}{\partial \beta}=0$ if and only if $\frac{\partial^2 \bar{G}}{\partial \bar\beta^2}=\frac{\partial \bar{G}}{\partial \bar\beta}=0$.
	Finally, observe that the diffeomorphism $x:\R^{2n}\to\R^{2n}$ takes $\E^s(L)$ to $\E^s(\bar{L})$.
\end{proof}

\p{The following definition is fundamental.}

\begin{definition}\label{oddstablegen} A generating family $G:\left( \R^{n}\times \R^{m}, (\beta_0,\lambda_0) \right)\to\R$ is {\it $\RR^{odd}$-stable} if, for any
	representative $G':V\to\R$ of $G$,  there exists a neighborhood $W$ of $G'$ in the $C^{\infty}$-topology (Whitney)
	and a neighborhood $V$ of $(\beta_0,\lambda_0)$ such that for any generating family
	$\bar{G}'\in W$, there exists $(\bar\beta_0,\bar{\lambda}_0)\in V$ such that $G$ and $\bar{G}$ are (fibred) $\RR^{odd}$-equivalent, $\bar{G}$
	being the germ of $\bar{G}'$ at $(\bar\beta_0,\bar{\lambda}_0)$.
\end{definition}

\p{And the following definition is suitable for computations}.

 \begin{notation}
\w{ Let $\EE_k$ be the ring of smooth function-germs $ ( \R^m,0) \to \R$.
We denote by $\EE_{n+m}^{n-even}$ the ring of smooth function-germs $f : (\R^n\times \R^m,0) \to \R$ such that $f(-\beta,\lambda)\equiv f(\beta,\lambda)$ ($f$ is even in $\beta$), by $\EE_{n+m}^{n-odd}$ the set of odd smooth function-germs $g : (\R^n, 0) \to (\R, 0)$ such that $g(-\beta,\lambda)\equiv -g(\beta,\lambda)$ ($g$ is odd in $\beta$), which has a module structure over $\EE_{n+m}^{n-even}$.}
 \end{notation}
 
 \begin{definition}\label{infstab}
 	\w{A generating family $G:\left( \R^{n}\times \R^{m}, 0) \right)\to\R$ is \emph{infinitesimally $\RR^{odd}$-stable} if
  \begin{equation}\label{inf-stab}
 	\EE_{n+m}^{n-odd} =\EE_{n+m}^{n-even}\left\{  \beta_j \frac{\partial G}{\partial \beta_i},  i,j=1...n \right\}+\EE_k\left\{ \frac{\partial G}{\partial \lambda_l}|_{\R^n\times\{0\}}, l=1...m \right\}
 \end{equation}} 
 \end{definition}

We now have the following main result, whose  (not too short) proof is presented  at the end of the paper, in Section 6.

\begin{theorem}\label{equivalence}
	Fibred-$\mathbb Z_2$-symmetric Lagrangian map-germs are  $\mathbb Z_2$-symmetrically Lagrangian equivalent (cf. Definition \ref{d3}) if and only if their odd generating families are  fibred  $\mathcal R^{odd}$-equivalent (cf. Definition \ref{foe}).
\end{theorem}

As a direct consequence of  Theorem \ref{equivalence} we have the following:

\begin{corollary}\label{mc} A  fibred-$\mathbb Z_2$-symmetric Lagrangian map-germ is $\mathbb Z_2$-symmetrically Lagrangian stable (cf. Definition \ref{d3b}) if and only if its odd generating family is {\it $\RR^{odd}$-stable} (cf. Definition \ref{oddstablegen}).
\end{corollary}

\begin{remark} As mentioned at the beginning of this section, the adaptation of Theorem \ref{equivalence} and Corollary \ref{mc} to the context of graph-like Legendrian map-germs is straightforward, once we  adapt definitions \ref{d2} and \ref{d3}-\ref{d3b} to the graph-like Legendrian setting as well. We refer the reader to \cite[Section 5.3]{Izu15} for a detailed thorough exposition of the straightforward relationship between the Lagrangian and Legendrian descriptions in terms of generating families when the Legendrian immersions are graph-like, which is always the case for the center-chord and special IAS obtained from a Lagrangian submanifold. \end{remark}

The final result for the classification of singularities on shell is given by:

 \begin{theorem}\label{oddstablegen2}
 	A generating family $G:\R^n\times\R^m\to\R$ is $\RR^{odd}$-stable
 	if and only if $G$ is an $\RR^{odd}$-versal deformation of $G_0=G(\cdot,0)$.
 \end{theorem}

Theorem  \ref{oddstablegen2} \p{follows, as a special case, from basic general theorems in singularity theory, as we now explain}. In the nonsymmetric case, the analogous to Theorem  \ref{oddstablegen2} can be divided in two theorems.

The first one states that a family $F:\mathbb R^n\times\mathbb R^m\to\mathbb R$  is stable (definition analogous to Definition \ref{oddstablegen} but replacing $\mathcal D_n^{odd}$ by the full diffeomorphism group $\mathcal D_n$ in Definition \ref{oe}) if and only if  $F$ is infinitesimally stable \p{(analogous to Definition \ref{infstab})}.
The concept of infinitesimal stability for $F$ under an action of a group $\mathcal G$ means, loosely speaking, that the $\mathcal G$-orbit of such an action contains a neighborhood of $F$.

A very important property of an infinitesimally stable family $F$  is its finite determinacy, meaning that $F$ is equivalent to $F'$  under this $\mathcal R^+$ action iff \m{there exists $k\in\mathbb N$
such that} $F$ and $F'$ are $\mathcal R^+$ equivalent up to the $k^{th}$ order in their Taylor expansions on $\mathbb R^n$.
Around 1968, J. Mather \cite{Mather} proved that infinitesimally stable families are stable, and vice versa if $F$ is proper.

At that same time Mather also proved the second theorem, which states that a family $F:\mathbb R^n\times\mathbb R^m\to\mathbb R$ is infinitesimally stable  if and only if $F$ is a versal deformation of $F_0=F(\cdot,0):\mathbb R^n\to\mathbb R$. The concept of a versal deformation  $F$ of a function $f:\mathbb R^n\to\mathbb R$ means, loosely speaking, that $F$ contains all possible deformations of $f$ or, more precisely, that any deformation $F'$ of $f$ is (fibred) $\mathcal R^+$ equivalent to one induced (by possibly eliminating some parameters) from $F$. If $F:\mathbb R^n\times\mathbb R^m\to\mathbb R$ is a versal deformation of $f=F_0$, then $\overline{F}:\mathbb R^n\times\mathbb R^m\to\mathbb R\times\mathbb R^m$, $\overline{F}(\beta,\lambda)=(F(\beta,\lambda),\lambda)$, is called a versal unfolding of $f$.
Finally, the versal deformations $F$ (or unfoldings $\overline{F}$) of $f$ with the least possible number of parameters  are called miniversal deformations (or unfoldings) of $f$, and they are all equivalent.

The complete statements and proofs of these two theorems belong to the basics of singularity theory, so they can be found in various texts, as for instance in \cite{Arnold} (see also \cite{Izu15}). In fact, these theorems are stated and proved for general (families of) maps, not just functions (as Mather did \cite{Mather}).

Then, around 1981, J. Damon \cite{Damon} (see also \cite{Damon2}) showed that these basic theorems of singularity theory are still valid if the appropriate group action induced from a full diffeomorphism group (for $F:\mathbb R^n\times\mathbb R^m\to\mathbb R$ this is the fibred $\mathcal R^+$ group action) is
 replaced by a subgroup satisfying certain properties (natural, tangential, exponential and filtrational), which he called a {\it geometrical subgroup}. The key point is that the (fibred) $\mathcal R^{odd}$ group action, induced from $\mathcal D_n^{odd}$  acting on odd families $F:\mathbb R^n\times\mathbb R^m\to\mathbb R$ as in Definition \ref{oe}, is a geometrical subgroup in this sense, so these basic theorems of singularity theory \m{go through}  \footnote{\p{In our odd setting, the first theorem asserts that Definitions \ref{oddstablegen} and \ref{infstab} are equivalent, while the second theorem  asserts that equation (\ref{inf-stab}) in Definition \ref{infstab} is equivalent to equation (\ref{inf-versal}) in Proposition \ref{versal}. This second theorem  is more direct, cf. Remark \ref{isver}.}},  implying the statement of Theorem \ref{oddstablegen2}.

 We end this section with the following results from \cite{DMR} which characterize the $\RR^{odd}$-versal deformations $G:\R^n\times\R^{m}\to\R$, \p{emphasizing} that, although the results  in \cite{DMR} were obtained in the smooth category, they also hold in the real analytic category (see \cite{BKP}). 
 
 \begin{notation}
We denote by $\EE_n^{even}$ the ring of even smooth function-germs $f : (\R^n,0) \to \R$ , by $\EE_n^{odd}$ the set of odd smooth function-germs $g : (\R^n, 0) \to (\R, 0)$, which has a module structure over $\EE_n^{even}$, and by $\M_n^{k(odd)}$, $k$ odd,  the $\EE_n^{even}$-submodule of $\EE_n^{odd}$  generated by $x^{k_1} \cdots x^{k_n}$,  s.t. $ k_i\geq 0$, $\sum k_i=k$.
 \end{notation}
 
 First, in the general case:

 \begin{proposition}{\rm (cf. \cite[Theorem 3.9]{DMR})}\label{versal} A $m$-parameter deformation $G(\beta,\lambda)$ of $G_0(\beta)=G(\beta,0)$ is $\RR^{odd}$-versal if and only if   
 	\begin{equation}\label{inf-versal}
 	\EE_n^{odd} =\EE_n^{even}\left\{  \beta_j \frac{\partial G_0}{\partial \beta_i},  i,j=1...n \right\}+\R\left\{ \frac{\partial G}{\partial \lambda_l}|_{\R^n\times\{0\}}, l=1...m \right\}.
 	\end{equation}
 \end{proposition}

Then, specifically for center-chord or special IAS on shell:
  \begin{corollary}{\rm (cf. \cite[Corollary 4.4, Theorem 4.5]{DMR})}
 	The germ of a generating family $G$ is an $\RR^{odd}$-versal deformation of $G_0$ if and only if
 	$$
 	\M_{n}^{3(odd)}=\EE_n^{even}\left\{  \beta_j \frac{\partial G_0}{\partial \beta_i},  i,j=1...n \right\}+\R\left\{ \frac{\partial g}{\partial q_l}|_{\R^n\times\{0\}}, l=1...n \right\},
 	$$
 	where the relation between $G$ ($=G_{cc}$ or \ $=G_{sp}$)  and $g$ ($=g_{cc}$ or \ $=g_{sp}$) are given by (\ref{gen-fam-gen-fun})-(\ref{eq:GenFamSP}).
 \end{corollary}
\begin{remark}\label{isver}
	\w{Using the  general Malgrange Preparation Theorem (see \cite{Martinet}, Chapter X, Section 6.3), one can show that Proposition \ref{versal} implies that infinitesimal $\RR^{odd}$-stability (\ref{inf-stab}) is equivalent to $\RR^{odd}$-versality (\ref{inf-versal}).}
\end{remark}

\section{Realization of simple singularities on shell for canonical IAS}\label{realizing}

\subsection{Simple singularities of odd functions}

The following results are a compilation of results in \cite{DMR}, section 3.

\w{We recall that the germ of a function $f$ has a simple singularity  if there exists \m{a} neighborhood of $f$ that intersects only \m{a} finite number of \m{equivalence classes of} other singularities  .}

Let $G_0\in \EE_n^{odd}$ with a singular point at $0$. If $n\geq 3$, $G_0$ is not $\RR^{odd}$-simple. For $n=1$, only the singularities $A_{2k/2}$ are simple. The corresponding miniversal deformations are
$$
G(t,\lambda_1,...,\lambda_k)=t^{2k+1}+\sum_{j=1}^k \lambda_jt^{2j-1}
$$
For codimension \m{$\leq 2$}, the only possibilities are $A_{2/2}$ and $A_{4/2}$.

For $n=2$, the following singularities are simple:

\begin{enumerate}
\item $D_{2k/2}^{\pm}$: $(t_1,t_2)\to t_1^2t_2+t_2^{2k-1}$, $k=2,3..$
\item $E_{8/2}$: $(t_1,t_2)\to t_1^3+t_2^5$
\item $J_{10/2}^{\pm}$: $(t_1,t_2)\to t_1^3+t_1t_2^4$
\item $E_{12/2}$:  $(t_1,t_2)\to t_1^3+t_2^7$
\end{enumerate}
The corresponding miniversal deformations are
\begin{enumerate}
\item  $D_{2k/2}^{\pm}$: $G(t_1,t_2,\lambda_1,..,\lambda_k)=t_1^2t_2+t_2^{2k-1}+\lambda_1t_1+\sum_{i=2}^k\lambda_it_2^{2i-3}$
\item $E_{8/2}$: $G(t_1,t_2,\lambda_1,...,\lambda_4)=t_1^3+t_2^5+\lambda_1t_1+\lambda_2t_2+\lambda_3t_1t_2^2+\lambda_4t_2^3$
\item $J_{10/2}^{\pm}$: $G(t_1,t_2,\lambda_1,...,\lambda_5)=t_1^3+t_1t_2^4+\lambda_1t_1+\lambda_2t_2+\lambda_3t_1^2t_2+\lambda_4t_1t_2^2+\lambda_5t_2^3$
\item $E_{12/2}$: $G(t_1,t_2,\lambda_1,...,\lambda_6)=t_1^3+t_2^7+\lambda_1t_1+\lambda_2t_2+\lambda_3t_1t_2^2+\lambda_4t_2^3+\lambda_5t_1t_2^4+\lambda_6t_2^5$
\end{enumerate}
For codimension $\leq 4$, the only possibilities are $D_{4/2}^{\pm}$, $D_{6/2}^{\pm}$, $D_{8/2}^{\pm}$ and $E_{8/2}$.

\subsection{Relation between the generating families for $\phi_{cc}(L)$ and $\phi_{sp}(L)$}

\w{Assume $L$ is locally the graph of $dS(s)$, with $S$ the germ at $0$ of a real analytic function,
$
S(s)=\sum_{(k_1,\cdots,k_n)\in \mathbb N^n} a_{k_1,\cdots,k_n}s_1^{k_1}\cdots s_n^{k_n}, \ \ s=(s_1,\cdots,s_n)\in \mathbb R^n 
$, 
where  $\mathbb N$ denotes the set of natural numbers including $0$ and $ a_{k_1,\cdots,k_n}\in \mathbb R$ for any $(k_1,\cdots,k_n)\in \mathbb N^n$.}
\w{Take then the germ at $0$ of a holomorphic function
$
H(z)=\sum_{(k_1,\cdots,k_n)\in \mathbb N^n} a_{k_1,\cdots,k_n}z_1^{k_1}\cdots z_n^{k_n}, \ \ z=(z_1,\cdots,z_n)\in \mathbb C^n.
$
}
\begin{lemma}
For $L\subset\R^{2m}$ a Lagrangian submanifold, let $g_{cc}(q,\dot{q})$ and $g_{sp}(q,\dot{q})$ denote the generating functions of $\phi_{cc}(L)$
and $\phi_{sp}(L)$. Then \footnote{\p{To simplify the notation, from now on we are dropping the superscript $s$ for the generating functions and generating families of the on-shell part of the IAS.}}
\begin{equation}\label{eq:RelationCCSP}
g_{cc}(q,\dot{q})=-ig_{sp}(q,i\dot{q}).
\end{equation}
In other words, 
$$
\mbox{if} \ \ g_{sp}(q,\dot{q})=\sum_{j=1,odd}^{\infty} b_j(q) (-1)^{\left\lfloor{j/2} \right\rfloor} (\dot{q})^j, \ \mbox{then} \ \ g_{cc}(q,\dot{q})=\sum_{j=1,odd}^{\infty} b_j(q) (\dot{q})^j.
$$
\end{lemma}
\begin{proof}
We have that
$$
Q(s,t)=\sum_{j=1,odd}^{\infty}  b_j(s) (-1)^{\left\lfloor{j/2} \right\rfloor} t^j, \ \mbox{where} \ b_j(s)=\sum_{k=j}^{\infty} a_k\binom{k}{j} s^{k-j}.
$$
On the other hand,
$$
g_{cc}(s,\beta)=\frac{1}{2}\sum_ka_k\left( (s+\beta)^k-(s-\beta)^k \right) \ \Rightarrow \ g_{cc}(s,\beta)=\sum_{j=1,odd}^{\infty}b_j(s) \beta^j.
$$
\end{proof}

\subsection{Simple singularities on shell of $\phi_{cc}(L)$ and $\phi_{sp}(L)$}

We now show by examples that, by an adequate choice of $L$, the simple singularities $A_{2/2}$, $A_{4/2}$, $D_{4/2}^{\pm}$, $D_{6/2}^{\pm}$,$D_{8/2}^{\pm}$ and $E_{8/2}$ appear as stable singularities of $\phi_{cc}(L)$ and $\phi_{sp}(L)$.

\begin{example}\label{ex:A22}
Consider $S(q)=q^3$. Then,
$$
G_{cc}(\beta,q,p)=\beta^3+3q^2\beta-p\beta.
$$
and
$$
G_{sp}(\beta,q,p)=-\beta^3+3q^2\beta-p\beta,
$$
which are versal unfoldings of $A_{2/2}$-singularities.
\end{example}

\begin{example}\label{ex:A42}
Consider $S(q)=q^5+\frac{1}{4}q^4$. Then,
$$
G_{cc}(\beta,q,p)=\beta^5 +10q^2\beta^3+5q^4\beta+q^3\beta+q\beta^3-p\beta.
$$
The caustic $\E_{cc}^s(L)$ is given by $\beta=0$ or $3q+30q^2+10\beta^2=0$ (see Figure \ref{fig:Caustics} (a)).
We have also
$$
G_{sp}=\beta^5-10q^2\beta^3+5q^4\beta+q^3\beta-q\beta^3-p\beta.
$$
Thus $\E_{sp}^s(L)$ is given by $\beta=0$ or
$3q=10\beta^2-30q^2$ (see Figure \ref{fig:Caustics} (b)). Observe that both constructions lead to versal unfoldings
of an $A_{4/2}$-singularity.
\end{example}

\begin{figure}[htb]
\centering \fsep \subfigure[ The set $\E_{cc}(L)$.] {
\includegraphics[width=.25\linewidth,clip =false]{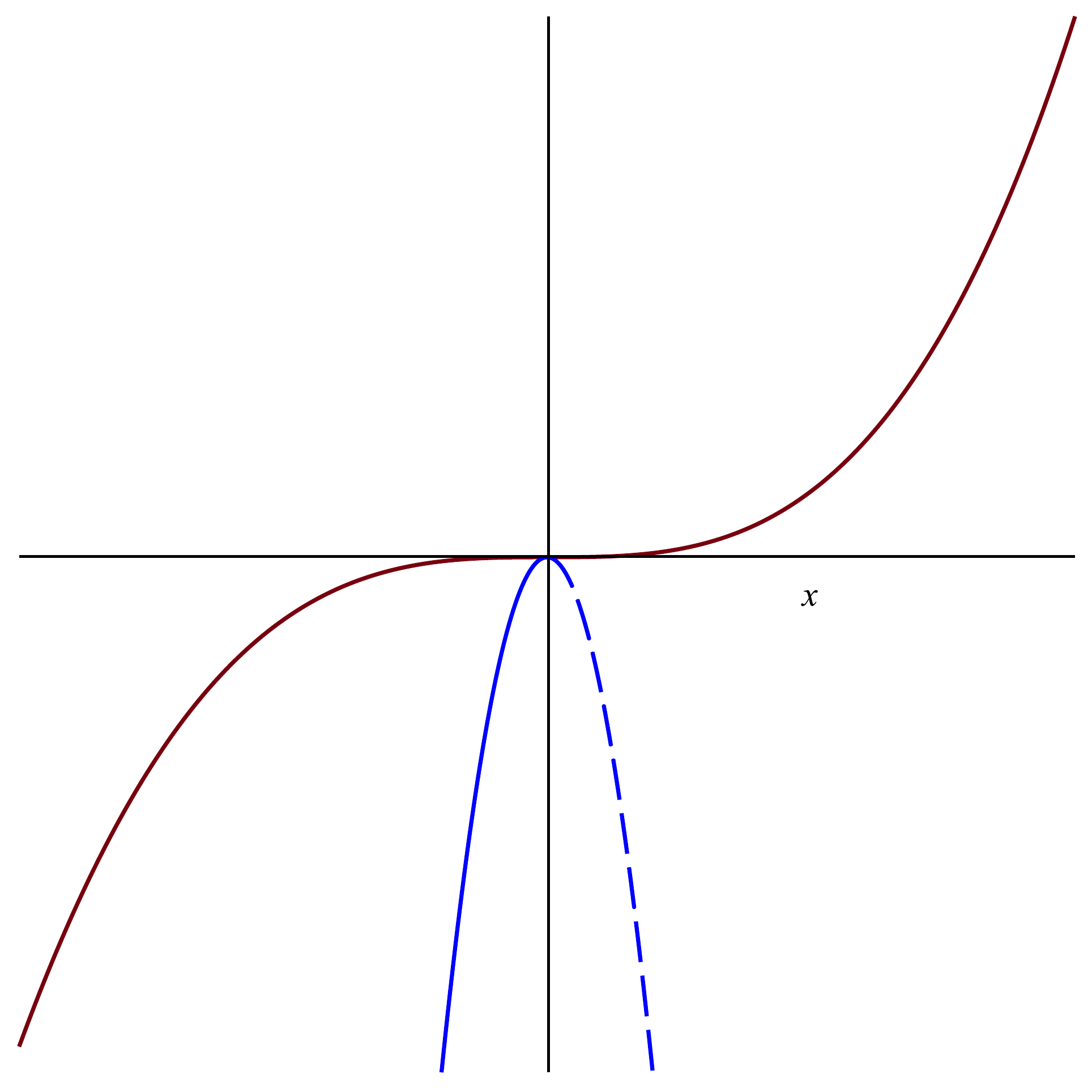}} \fsep\subfigure[
The set $\E_{sp}(L)$. ] {
\includegraphics[width=.25\linewidth,clip
=false]{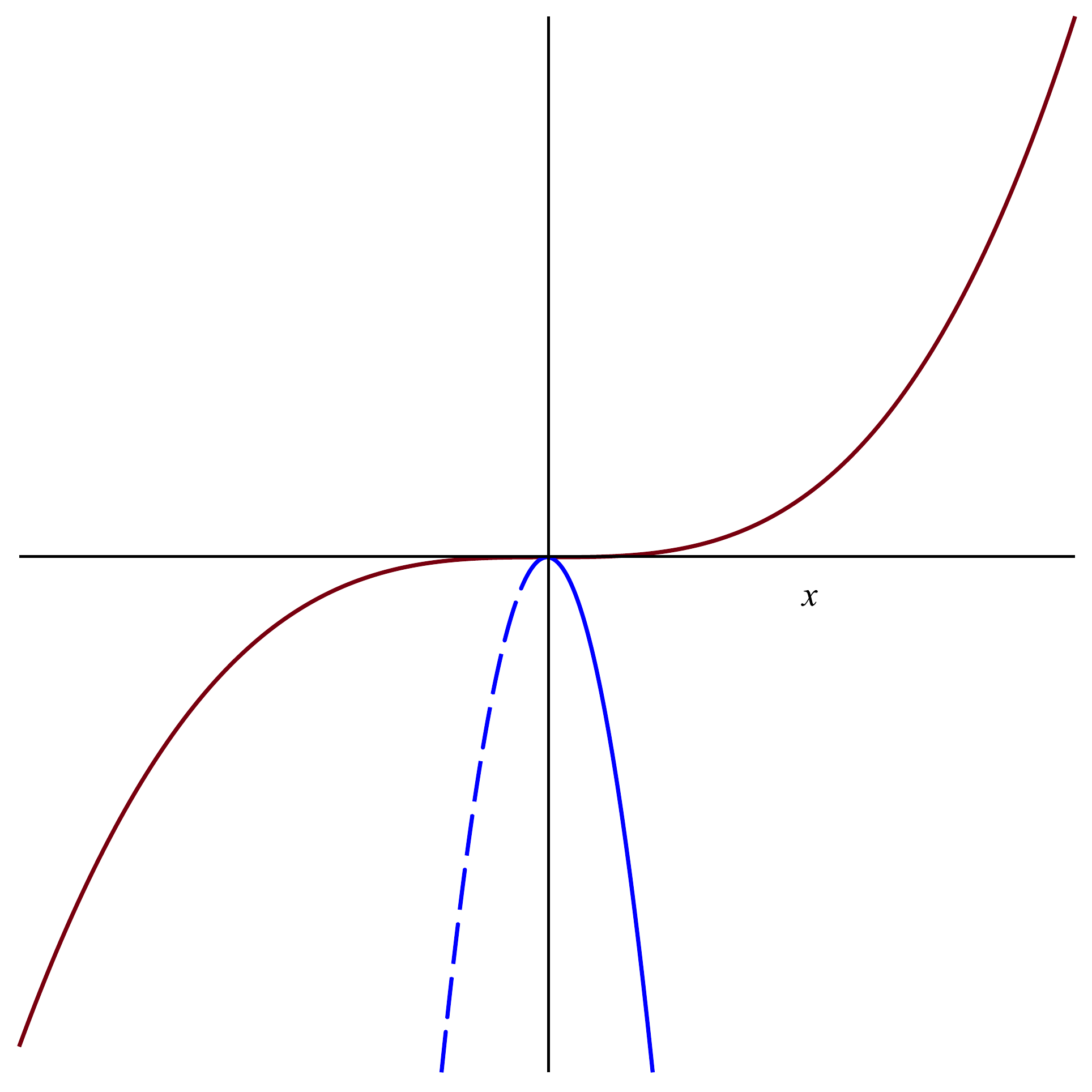}}\fsep
\caption{The caustics $\E_{cc}(L)$ and $\E_{sp}(L)$ of example \ref{ex:A42}.}
\label{fig:Caustics}
\end{figure}

\begin{example}\label{ex:D42}
Consider \  
$
S(u_1,u_2)=u_1^2u_2\pm u_2^3
$. \
Then,
$$
G_{cc}(\beta,q,p)=\pm \beta_2^3 +\beta_1^2\beta_2-p_1\beta_1-p_2\beta_ 2\pm3q_2^2\beta_2+q_1^2\beta_2+2q_1q_2\beta_1
$$
and the singular set is defined by $\pm 3\beta_2^2=\beta_1^2$.
In the special case, we have
$$
G_{sp}(\beta,q,p)=\mp \beta_2^3 -\beta_1^2\beta_2-p_1\beta_1-p_2\beta_ 2\pm3q_2^2\beta_2+q_1^2\beta_2+2q_1q_2\beta_1,
$$
and the singular set is again defined by $\pm 3\beta_2^2=\beta_1^2$. Both constructions lead to versal unfoldings of a
$D_{4/2}^{\pm}$-singularity.
\end{example}

\begin{example}\label{ex:D62}
Consider \ 
$
S=q_1^2q_2\pm q_ 2^5+\frac{1}{4}q_2^4.
$ \ 
Then, 
$$
G_{cc}=\beta_1^2\beta_2\pm\beta_2^5+q_2\beta_2^3-p_1\beta_1-p_2\beta_2+q_1^2\beta_2+2q_1q_2\beta_1\pm 10\beta_2^3q_2^2\pm 5\beta_2q_2^4+\beta_2q_2^3,
$$
while
$$
G_{sp}=-\beta_1^2\beta_2\pm\beta_2^5-q_2\beta_2^3-p_1\beta_1-p_2\beta_2+q_1^2\beta_2+2q_1q_2\beta_1\mp 10\beta_2^3q_2^2\pm 5\beta_2q_2^4+\beta_2q_2^3.
$$
Thus, $G_{cc}$ is a versal unfolding of a $D_{6/2}^{\pm}$-singularity, while
$G_{sp}$ is a versal unfolding of a $D_{6/2}^{\mp}$ singularity.
\end{example}

\begin{example}\label{ex:D82}
Consider \ 
$
S=q_1^2q_2\pm q_ 2^7+q_1q_2^3+\frac{1}{6}q_2^6.
$ \newline 
Then, the corresponding $G_{cc}$ is a versal unfolding of a $D_{8/2}^{\pm}$ singularity.
\newline The corresponding $G_{sp}$ is a versal unfolding of a $D_{8/2}^{\pm}$ singularity.
\end{example}

\begin{example}\label{ex:E82}
Consider \ 
$
S=q_1^3+q_ 2^5+q_1q_2^3.
$ \newline 
Then, the corresponding $G_{cc}$ is a versal unfolding of a $E_{8/2}$ singularity.
\newline The corresponding $G_{sp}$ is a versal unfolding of a $E_{8/2}$ singularity.
\end{example}

\begin{remark}
In examples \ref{ex:A22} and \ref{ex:A42}, $\E_{cc}^s(L)$ and $\E_{sp}^s(L)$ are diffeomorphic, since they are bifurcation sets of points $A_{2/2}$ and $A_{4/2}$, respectively. The same occurs in examples \ref{ex:D42}, \ref{ex:D82} and \ref{ex:E82}, but not in example \ref{ex:D62}.
\end{remark}

\subsection{Stable singularities on shell for the two canonical IAS obtained from a given Lagrangian curve or surface}

We now classify all $\RR^{odd}$-stable singularities that appear in the caustics $\E^s(L)$, $\tilde\E^s(L)$, when
$L$ is a planar curve or a Lagrangian surface in $\R^4$. In these dimensions, only simple singularities are stable, so we apply the previous results taking care of the possible codimensions.

\subsubsection{Lagrangian Curves} We follow section 4.1 of \cite{DMR}. Let $L$ be a germ at $0$ of a curve and assume that $L$
is generated by a function germ $S\in\M_1^3\subset\EE_1$.

\begin{proposition}\label{stablecurve}
	\
\begin{enumerate}
\item\label{cond1c} If $S^{(3)}(0)\neq 0$, $G_{cc}$ and $G_{sp}$ are $\RR^{odd}$-equivalent to the $\RR^{odd}$ versal deformation of  $A_{2/2}$.
\item\label{cond2c} If $S^{(3)}(0)=0$, $S^{(4)}(0)\neq 0$, $S^{(5)}(0)\neq 0$, $G_{cc}$ and $G_{sp}$ are $\RR^{odd}$-equivalent to the $\RR^{odd}$ versal
deformation of $A_{4/2}$.
\end{enumerate}
\end{proposition}

\begin{proof}
If $S^{(3)}\neq 0$, then $f^{(3)}\neq 0$ and $g^{(3)}\neq 0$. Thus $G_{cc}$ and $G_{sp}$ are odd deformations of an $A_{2/2}$-singularity, and it is easy to see
that they are in fact versal deformations of $G_0$. Thus we have proved item \ref{cond1c}.

For item \ref{cond2c}, observe that the hypothesis imply that $F$ and $G$ are odd deformations of an $A_{4/2}$-singularity. It is also easy to verify
that this deformation is versal, thus proving the result.
\end{proof}

The geometric interpretation of condition \ref{cond1c} is that the curvature of $L$ does not vanish, while the geometric interpretation
of condition \ref{cond2c} is that the curvature vanishes, but its first and second derivatives do not.

\begin{corollary}\label{EL=L}
If $L$ is strongly convex, then $\E_{cc}^s(L)=\E_{sp}^s(L)=L$.
\end{corollary}
\begin{remark}
	Example \ref{ex:TorusIn} (cf. Figure 1) and Example \ref{ex:TorusOut} (cf. Figure 2) are particular nongeneric illustrations of the above corollary, while Example \ref{ex:A22}  illustrate the generic case, locally.  The generic case when condition $2$ of Proposition \ref{stablecurve} is satisfied is illustrated by Example \ref{ex:A42} (cf. Figure \ref{fig:Caustics}).
\end{remark}

\subsubsection{Lagrangian Surfaces} We follow section 4.2 of \cite{DMR}. Let $L$ be a germ at $0$ of a Lagrangian surface and assume that $L$
is generated by a function germ $S\in\M_2^3\subset\EE_2$.


\begin{notation}
Denote
$$
S_{i,j}=\frac{\partial^{i+j}S}{\partial q_1^i\partial q_2^j}(0,0),
$$
with
$$
j_0^3S=\frac{1}{6}S_{3,0}q_1^3+\frac{1}{2}S_{2,1}q_1^2q_2+\frac{1}{2}S_{1,2}q_1q_2^2+\frac{1}{6}S_{0,3}q_2^3
$$
denoting the $3$-jet of $S$ at $0$.  The discriminant of $j_0^3S$ is
\begin{equation*}
\Delta(j_0^3S)=\frac{1}{48}\left( 3S_{1,2}^2S_{2,1}^2-4S_{0,3}S_{2,1}^3-4S_{1,2}^3S_{3,0}-S_{0,3}^2S_{3,0}^2+6S_{0,3}S_{1,2}S_{2,1}S_{3,0}    \right).
\end{equation*}
\end{notation}

\begin{proposition}\label{first}
Assume $\Delta(j_0^3S)\neq 0$.
\begin{enumerate}
\item If $\Delta(j_0^3S)>0$, $G_{cc}$ and $G_{sp}$ are $\RR^{odd}$-equivalent to the $\RR^{odd}$ versal deformation of $D_{4/2}^{-}$.

\item If $\Delta(j_0^3S)<0$, $G_{cc}$ and $G_{sp}$ are $\RR^{odd}$-equivalent to the $\RR^{odd}$ versal deformation of $D_{4/2}^{+}$.
\end{enumerate}
\end{proposition}

\begin{proof}
Assume $\Delta(j_0^3S)>0$. Then, by a linear change of coordinates, we can write $j_0^3g=\beta_1^2\beta_2-\beta_2^3$. Thus $g$
is $\RR^{odd}$-equivalent to a $D_{4/2}^{-}$ singularity, and it is easy to see that $G_{sp}$ is an $\RR^{odd}$ versal deformation of $g$. This proves the first assertion for $G_{sp}$, the second one being similar. The proofs for $G_{cc}$ are similar or else one can invoke theorem 4.11 of \cite{DMR}.
\end{proof}


\begin{notation}
	Denote
	$$
	\delta_1=S_{3,0}S_{1,2}-S_{2,1}^2;\ \ \ \delta_2=S_{0,3}S_{2,1}-S_{1,2}^2.
	$$
	$$
	r_1=\frac{S_{2,1}S_{1,2}-S_{3,0}S_{0,3}}{2(S_{3,0}S_{1,2}-S_{2,1}^2)};\ \ r_2=\frac{S_{3,0}^2S_{0,3}-4S_{3,0}S_{2,1}S_{1,2}+3S_{2,1}^3}{S_{3,0}S_{1,2}-S_{2,1}^2}
	$$
	$$
	\sigma_{0,n}=\frac{\sum_{k=0}^n \binom{n}{k} S_{k,n-k}r_1^k}{(S_{3,0}r_1-r_2)^n},\ \ n=5,7.
	$$
	$$
	\tilde{r}_1=\frac{S_{2,1}S_{1,2}-S_{3,0}S_{0,3}}{2(S_{0,3}S_{2,1}-S_{1,2}^2)};\ \ \tilde{r}_2=\frac{S_{0,3}^2S_{3,0}-4S_{0,3}S_{2,1}S_{1,2}+3S_{1,2}^3}{S_{0,3}S_{2,1}-S_{1,2}^2}
	$$
	$$
	\sigma_{n,0}=\frac{\sum_{k=0}^n \binom{n}{k} S_{k,n-k}\tilde{r}_1^k}{(S_{0,3}\tilde{r}_1-\tilde{r}_2)^n},\ \ n=5,7.
	$$
	\end{notation}

\begin{lemma}
If $\Delta(j_0^3S)=0$, then $\delta_i\leq 0$, $i=1,2$.
\end{lemma}

\begin{proposition}\label{second}
Assume $\Delta(j_0^3S)=0$.
\begin{enumerate}
\item If $\delta_1\cdot\sigma_{0,5} <0$ or $\delta_2\cdot\sigma_{5,0}<0$, $G_{cc}$ is $\RR^{odd}$-equivalent to the $\RR^{odd}$ versal deformation of  $D_{6/2}^{+}$, while $G$ is $\RR^{odd}$-equivalent to the $\RR^{odd}$ versal deformation of  $D_{6/2}^{-}$.

\item If $\delta_1\cdot\sigma_{0,5}>0$ or $\delta_2\cdot\sigma_{5,0}>0$, $F$ is $\RR^{odd}$-equivalent to the $\RR^{odd}$ versal deformation $D_{6/2}^{-}$,
while $G_{sp}$ is $\RR^{odd}$-equivalent to the $\RR^{odd}$ versal deformation of  $D_{6/2}^{+}$.
\end{enumerate}
\end{proposition}

\begin{proof}
Similar to theorem 4.14 of \cite{DMR}.
\end{proof}

\begin{proposition}
Assume $\Delta(j_0^3S)=0$.
\begin{enumerate}
\item If $\delta_1<0$, $\sigma_{0,5}=0$ and $\sigma_{0,7}>0$ or $\delta_2<0$, $\sigma_{5,0}=0$ and $\sigma_{7,0}>0$, $G_{cc}$ and $G_{sp}$ are $\RR^{odd}$-equivalent to the $\RR^{odd}$ versal deformation of $D_{8/2}^{+}$.

\item If $\delta_1<0$, $\sigma_{0,5}=0$ and $\sigma_{0,7}<0$ or $\delta_2<0$, $\sigma_{5,0}=0$ and $\sigma_{7,0}<0$, $G_{cc}$ and $G_{sp}$ are $\RR^{odd}$-equivalent to the $\RR^{odd}$ versal deformation of $D_{8/2}^{-}$.
\end{enumerate}
\end{proposition}

\begin{proof}
Similar to theorem 4.15 of \cite{DMR}.
\end{proof}

\begin{proposition}\label{last}
Assume $\Delta(j_0^3S)=0$.
If
$$
\delta_1=0,\ \ S_{3,0}\neq 0, \sum_{k=0}^5\binom{5}{k}S_{k,5-k}(-S_{2,1})^k(S_{3,0})^{5-k}\neq 0,
$$
or
$$
\delta_2=0,\ \ S_{0,3}\neq 0, \sum_{k=0}^5\binom{5}{k}S_{k,5-k}(-S_{1,2})^k(S_{0,3})^{5-k}\neq 0,
$$
then $G_{cc}$ and $G_{sp}$ are $\RR^{odd}$-equivalent to the $\RR^{odd}$ versal deformation of $E_{8/2}$.
\end{proposition}

\begin{proof}
Similar to theorem 4.16 of \cite{DMR}.
\end{proof}

\begin{remark}
	For detailed geometric interpretations of all the conditions of propositions \ref{first}-\ref{last}, we refer to section 4.3 of \cite{DMR}. Here, we just point out that conditions of Proposition \ref{first} are realized for hyperbolic and elliptic points of $L$, the higher singularities of propositions \ref{second}-\ref{last} occurring for parabolic points of $L$. In particular, the local equivalent of Corollary \ref{EL=L} is realized for hyperbolic points of $L$, that is, if $L'$ is the germ of $L$ at a generic hyperbolic point of $L$, then $\E_{cc}^s(L')=\E_{sp}^s(L')=L'$  (cf. \cite[Corollary 4.19]{DMR}).
\end{remark}

\section{Proof of Theorem \ref{equivalence}}\label{App}

We now prove Theorem \ref{equivalence}, which relates the definition of equivalence of fibred-$\mathbb Z_2$-symmetric Lagrangian, resp. Legendrian, map-germs (cf. Definition \ref{d3}) to the definition of fibred $\mathcal R^{odd}$-equivalence of their odd generating families (cf. Definition \ref{foe}). 	We prove this theorem by modifying the method used in \cite[Section 19.5]{Arnold} to the case of $\mathbb Z_2$-symmetric Lagrangian equivalence.
	
	First, assume that odd generating families $F_1$ and $F_2$ are  fibred  $\mathcal R^{odd}$-equivalent, cf. Definition \ref{foe}. 
	Then, the fibred diffeomorphism $(\Psi^{-1})^{\ast}$  of the big phase space $T^{\ast} \mathbb R^{n+m}$ determines  a Lagrangian equivalence of the big phase space between Lagrangian sections of $ T^{\ast} \mathbb R^{n+m}$ generated by the function-germs $F_1$ and $F_2$ on the big space. Both Lagrangian sections are $\rho$-regular. Since the diffeomorphism-germ $\Psi$ of $\mathbb R^{n+m}$ is fibred, the Lagrangian equivalence of the big phase space induces a Lagrangian equivalence of the small phase space $T^{\ast}\mathbb R^m$ (see \cite{Arnold}, Section 19.4) between germs of Lagrangian submanifolds generated by the odd families $F_1$ and $F_2$. 
	It is easy to check that this Lagrangian equivalence of the small phase space $T^{\ast}\mathbb R^m$ is determined by the diffeomorphism-germ $(\Lambda^{-1})^{\ast}$. But   $(\Lambda^{-1})^{\ast}$ is a linear map in the fibers of $T^{\ast}\mathbb R^{m}$ , hence is odd in the fibers. Thus, the Lagrangian map-germs generated by odd families $F_1$ and $F_2$ are  $\mathbb Z_2$-symmetrically Lagrangian equivalent.
	
	Now, assume that we are given a $\mathbb Z_2$-symmetric Lagrangian equivalence of the small phase space, mapping the germ of a fibred-$\mathbb Z_2$-symmetric Lagrangian submanifold  $L_1$, determined by an odd generating family $F_1$, to the germ of a fibred-$\mathbb Z_2$-symmetric Lagrangian submanifold $L_2$. By Proposition \ref{odd-base},  the $\mathbb Z_2$-symmetric Lagrangian equivalence is determined by $(\phi)^{\ast}$, where $\phi$ is a diffeomorphism-germ of the base. Let us consider a diffeomorphism-germ $(Id_{ \mathbb R^n}, \phi):\mathbb R^{n}\times \mathbb R^m\rightarrow
	\mathbb R^{n}\times \mathbb R^m$ of the big space. This induces a Lagrangian equivalence $(Id_{ \mathbb R^n}, \phi)^{\ast}$ of the big phase space, mapping the germ of the Lagrangian section $\mathcal L_1$ generated by  function-germ $F_1$ to the germ of the Lagrangian section $\mathcal L_2$.  Then, it is easy to see that $\mathcal L_2$ is generated by a function germ $F_2$ of the form $F_2(\beta,\lambda)=F_1(\beta,\phi^{-1}(\lambda))$. It implies that $F_2$ is an odd generating family of $L_2$ which is fibered  $\mathcal R^{odd}$-equivalent to $F_1$.

	To finish the proof of Theorem \ref{equivalence} we need the following lemmas. But first some preparations. 
	
	By Remark \ref{pathological}, every fibred-$\mathbb Z_2$-symmetric Lagrangian germ admits a generating function-germ $S=S(\kappa_J,\lambda_I)$  which is odd in $\kappa_J$, with the minimal number of pathological arguments $\kappa_J$. This number is equal to  the dimension of the kernel of the differential of the Lagrangian map-germ. We fix the set $\kappa_J$ of $n$ pathological arguments. By Remark \ref{map_by_F}, we obtain that the Lagrangian map-germ is given in terms of $F$ by $\Sigma(F)\ni (\beta,\lambda)\mapsto \lambda\in R^m$.

	A vector $\eta$ is tangent to $\Sigma(F)$ at $(0,0)$ if $d(\frac{\partial F}{\partial \beta})|_{(0,0)}(\eta)=0$ and $\eta$ is in the kernel of the Lagrangian germ if  $d\lambda|_{(0,0)}(\eta)=0$. This implies that $d\beta(\eta)$ is in the kernel of the map $\frac{\partial^2 F}{\partial \beta^2}(0,0):\mathbb R^n\rightarrow \mathbb R^n$. But since $F$ is odd, $\frac{\partial^2 F}{\partial \beta^2}(-\beta,\lambda)\equiv -\frac{\partial^2 F}{\partial \beta^2}(\beta,\lambda)$, thus $	\frac{\partial^2 F}{\partial \beta^2}(0,\lambda)\equiv 0$ and therefore 
	the image of the kernel of the Lagrangian map under the linear map $d\beta:\mathbb R^{n+m}\rightarrow \mathbb R^n$  is the whole space $\mathbb R^n$.  But $\kappa_J$ are fixed $n$ pathological arguments. Thus the image of  the kernel of the Lagrangian map under the linear map $d\kappa_J=d(\frac {\partial F}{\partial \lambda_J})(0,0):\mathbb R^{n+m}\rightarrow \mathbb R^m$ is $n$-dimensional. But we have $d\kappa_J(\eta)= \frac{\partial^2 F}{\partial \beta\partial \lambda_J}(0,0)d\beta(\eta)$, because $d\lambda_J(\eta)=0$.  Hence, if $F$ is odd then
	\begin{equation}\label{lambda_J}
	\det \frac{\partial^2 F}{\partial \beta\partial \lambda_J}(0,0)\ne 0.
	\end{equation}
	
	An odd generating family $F$ is called {\it special} if for $\frac{\partial F}{\partial \beta}(\beta,\lambda)=0$ the condition $\beta=\frac{\partial F}{\partial \lambda_J}(\beta,\lambda)$ is fulfilled.
	We then have the following lemmas.
	\begin{lemma}\label{to_special}
		The germ of an odd generating family is fibred $\mathcal R^{odd}$-equivalent to the germ of a special odd generating family of the same Lagrangian germ.
	\end{lemma}
	\begin{proof}[Proof of Lemma \ref{to_special}]
		We follow the proof of Lemma 1 in Section 19.5 \cite{Arnold}. Since  $F$ is odd the condition (\ref{lambda_J}) is fulfilled. Hence the map-germ $\Psi(\beta,\lambda)\equiv ( \frac{\partial F}{\partial \lambda_J}(\beta,\lambda),\lambda)$ is a fibred diffeomorphism-germ of the big space. Since $F$ is odd in $\beta$, $\Psi$ is also odd in $\beta$. The germ $F$ is fibred $\mathcal R^{odd}$-equivalent to the germ of an odd generating family  $F_1(\beta,\lambda)\equiv F(\Psi^{-1}(\beta,\lambda))$.  It is easy to check that $F_1$ is special and it generates the same Lagrangian germ (see \cite{Arnold} for details).
	\end{proof}
	
	\begin{lemma}\label{special_equiv}
		The germs of special odd generating families, determining the same fibred-$\mathbb Z_2$-symmetric Lagrangian germ, are fibered $\mathcal R^{odd}$-equivalent.
	\end{lemma}
	\begin{proof}[Proof of Lemma \ref{special_equiv}]
		From  \cite[Section 19.5 (D)(d)]{Arnold}, any two special generating families $F_0$, $F_1$ of the same Lagrangian germ have the same set of critical points $\Sigma(F_0)=\Sigma(F_1)=\Sigma$, the restrictions of $F_0$ and $F_1$ to $\Sigma$ coincide up to an additive constant and the total differential of $F_0-F_1$ is equals to $0$ on the whole of $\Sigma$.
		Let $F_0$ and $F_1$ be two special odd generating families of the same   fibred-$\mathbb Z_2$-symmetric Lagrangian germ. Since $F_0, F_1$ are odd, we have $F_0(0,0)=F_1(0,0)=0$. Thus $F_0-F_1$ has zero of not less than second order on $\Sigma$. We use the homotopy method. 
		Let $F_t=F_0+t(F_1-F_0)$ for $t\in [0,1]$. Then $F_t$ is a special odd generating family of the same Lagrangian germ. We shall find a family $\Psi_t(\beta,\lambda)\equiv(\Phi_t(\beta,\lambda),\lambda)$ of odd diffeomorphisms in $\beta$, smoothly depending on $t\in [0, 1]$, such that
		\begin{equation}\label{eq_t}
		F_t\circ \Psi_t=F_0, \ \ \Psi_0=Id_{\mathbb R^{n+m}}. 
		\end{equation}
		The diffeomorphism-germ $\Psi_1$ establishes fibred $\mathcal R^{odd}$-equivalence of $F_0$ and $F_1$.    Differentiating (\ref{eq_t}) with respect to $t$ we obtain the equation
		\begin{equation}
		F_1(\beta,\lambda)-F_0(\beta,\lambda)+\sum_{i=1}^n \xi_i(\beta, \lambda, t) \frac{\partial F_t}{\partial \beta_i}(\beta,\lambda).
		\end{equation}
	
		Let $\Theta:\mathbb R^{n+m}\times [0,1]\rightarrow \mathbb R^{n+m}\times [0,1] $ be the  map-germ $\Theta(\beta,\lambda_J,\lambda_I, t)\equiv (\beta,\frac{\partial F_t}{\partial \beta}(\beta,\lambda),\lambda_I,t)$. Since $F_t$ is odd, $\frac{\partial F_t}{\partial \beta}$ is even and (\ref{lambda_J}) holds. Hence, 
		$\Theta$ is a diffeomorphism-germ, $\Theta^{-1}(\beta, v, \lambda_I,t)\equiv (\beta, \gamma(\beta,v,\lambda_I,t),\lambda_I,t)$ and $\gamma(-\beta,v,\lambda_I,t)\equiv \gamma (\beta,v,\lambda_I,t)$.
		Let $H$ be the family of function-germs on $\mathbb R^{n+m}$
		$$
		H(\beta,v,\lambda_I,t)\equiv F_0(\beta, \gamma(\beta,v,\lambda_I, t),\lambda_I)-F_0(\beta, \gamma(\beta,v,\lambda_I,t),\lambda_I).
		$$
		Then $H$ is odd in $\beta$ and $H(\beta,\frac{\partial F_t}{\partial \beta}(\beta,\lambda),\lambda_I,t)\equiv F_0(\beta,\lambda)-F_1(\beta,\lambda)$. It implies that $H(\beta,0,\lambda_I,t)\equiv (F_0(\beta,\lambda)-F_1(\beta,\lambda))|_{\Sigma}\equiv 0$. Let $h(s)=H(\beta,sv,\lambda_I,t)$ for $s\in [0,1]$. Thus $h(1)-h(0)=\int_{0}^1 \frac{dh}{ds}(s)ds$. Hence
		$$
		H(\beta, v,\lambda_I,t)\equiv \sum_{i=1}^n v_i\int_0^1\frac{\partial H}{\partial v_i}(\beta,sv,\lambda_I,t)ds.
		$$
		If we put $v=\frac{\partial F_t}{\partial \beta}(\beta,\lambda)$ we get
		\begin{equation}\label{F_0-F_1}
		F_0(\beta,\lambda)-F_1(\beta,\lambda)\equiv \sum_{i=1}^n \xi_i(\beta,\lambda,t)\frac{\partial F_t}{\partial \beta_i}(\beta,\lambda),
		\end{equation}
		where $\xi_i(\beta,\lambda,t)\equiv\int_0^1\frac{\partial H}{\partial v_i}(\beta,s\frac{\partial F_t}{\partial \beta}(\beta,\lambda),\lambda_I,t)ds$ for $i=1,\cdots,n$. It is easy to see that $\xi_i(-\beta,\lambda,t)=-\xi_i(\beta,\lambda,t)$.
		Since the total differential of $F_0-F_1$ vanishes on the whole of $\Sigma=\Sigma(F_t)$, by (\ref{F_0-F_1}) we have  $\xi_i|_{\Sigma}=0$ for $i=1,\cdots,n$.
		
		Thus, the vector field $\xi(\beta,\lambda,t)=\sum_{i=1}^n \xi_i(\beta,\lambda,t) \frac{\partial}{\partial \beta_i}$ depending on $t$  takes value $0$ on $\Sigma$ and  is odd  in $\beta$.
		Hence $\xi$ induces a  diffeomorphism $\Psi_t$ in the neighborhood of $(0,0)$ for all $t\in [0,1]$, which satisfies the  ODE system
		\begin{equation}\label{ODE}
		\frac{d\Psi_t}{dt}=\xi(\Psi_t).
		\end{equation}
		
		From the form of $\xi(\beta,\lambda,t)$,  the diffeomorphism $\Psi_t$ has  form $\Psi_t(\beta,\lambda)\equiv (\Phi_t(\beta,\lambda),\lambda)$.
		The maps $y(t)\equiv \Psi(-\beta,\lambda)$ and $z(t)\equiv -\Psi_t(\beta,\lambda)$ satisfy the system (\ref{ODE}) with the same initial condition $y(0)=z(0)=(-\beta,\lambda)$. By the uniqueness of the solution of the initial value problem,  $\Psi(-\beta,\lambda)\equiv -\Psi(\beta,\lambda)$.  Hence $\Psi_t$ is odd in $\beta$.
		Thus, the fibered odd diffeomorphism-germ $\Psi_1$ satisfies $F_1\circ \Psi_1=F_0$. Consequently, $F_1$ and $F_0$ are fibered $\mathcal R^{odd}$-equivalent.
	\end{proof}
	Then,  by Lemmas \ref{to_special}-\ref{special_equiv} we obtain that any two odd generating families of the same fibred-$\mathbb Z_2$-symmetric Lagrangian submanifold-germ $L$ are fibred  $\mathcal R^{odd}$-equivalent, which finishes the proof of Theorem \ref{equivalence}.

\section*{Acknowledgements}

We thank M.A.S. Ruas for discussions.


\end{document}